\documentclass[a4paper,10pt,leqno]{amsart}
\usepackage[total={6in,9in},
top=1in, left=1in, right=1in, bottom=1in]{geometry}
\usepackage{amsmath}
\usepackage[italian, english]{babel}
\usepackage{amsfonts}
\usepackage{amssymb}
\usepackage{dsfont}
\usepackage{mathrsfs}
\usepackage{graphicx}
\usepackage{setspace}
\usepackage{fancyhdr}
\usepackage{amsthm}
\usepackage{empheq}
\usepackage{cases}
\usepackage[all]{xy}
\usepackage{stmaryrd}
\usepackage{color}
\newcommand{\erre}{\mathbb{R}}

\newcommand{\enne}{\mathbb{N}}

\newcommand{\diver}{\operatorname{div}}

\newcommand{\ra}{\rightarrow}

\newcommand{\set}[1]{{\left\{#1\right\}}}               % tra parentesi graffe
\newcommand{\pa}[1]{{\left(#1\right)}}                  % tra tonde
\newcommand{\sq}[1]{{\left[#1\right]}}                  % tra quadre
\newcommand{\abs}[1]{{\left|#1\right|}}                 % valore assoluto
\newcommand{\pair}[1]{\left\langle#1\right\rangle}      % pairing

\newcommand{\eps}{\varepsilon}                           % eps
\newcommand{\intst}[1]{\int_0^\infty\int_M #1 \,d\mu dt}

\newtheorem{theorem}{\textbf{Theorem}}[section]
\newtheorem{lemma}[theorem]{\textbf{Lemma}}

\newtheorem{cor}[theorem]{\textbf{Corollary}}
\newtheorem{defi}[theorem]{\textbf{Definition}}
\theoremstyle{remark}
\newtheorem{rem}[theorem]{\textbf{Remark}}

\newtheorem{exe}[theorem]{\textbf{Example}}
\numberwithin{equation}{section}

\title[]
{Nonexistence of solutions to parabolic differential inequalities
with a potential on Riemannian manifolds}

\date{\today} 

\keywords{Parabolic inequalities on manifolds; weighted volume
growth; nonexistence of solutions}

\subjclass[2010]{35K59; 35K92; 35R01; 53C20}

\begin{document}

\maketitle

\begin{center}
\textsc{\textmd{P. Mastrolia\footnote{Universit\`{a} degli Studi di
Milano, Italy. Email: paolo.mastrolia@unimi.it.}, D. D.
Monticelli\footnote{Universit\`{a} degli Studi di Milano, Italy.
Email: dario.monticelli@unimi.it.} and F.
Punzo\footnote{Universit\`{a} degli Studi di Milano, Italy. Email:
fabio.punzo@unimi.it. \\ The three authors are supported by GNAMPA
project ``Analisi globale ed operatori degeneri'' and are members of
the Gruppo Nazionale per l'Analisi Matematica, la Probabilit\`{a} e
le loro Applicazioni (GNAMPA) of the Istituto Nazionale di Alta
Matematica (INdAM). }, }}
\end{center}

\begin{abstract}
We are concerned with nonexistence results of nonnegative weak
solutions for a class of quasilinear parabolic problems with a
potential on complete noncompact Riemannian manifolds. In
particular, we highlight the interplay between the geometry of the
underlying manifold, the power nonlinearity and the behavior of the
potential at infinity.
\end{abstract}

%\tableofcontents

\section{Introduction}\label{intro}

In this paper we investigate the nonexistence of nonnegative,
nontrivial weak solutions (in the sense of Definition \ref{def1}
below) to parabolic differential inequalities of the type
\begin{equation}\label{Eq}
\begin{cases}
  \partial_t u -\diver\pa{\abs{\nabla u}^{p-2}\nabla u}\geq V(x, t) u^q &\text{ in }\, M\times (0, \infty) \\ u = u_0 &\text{ in }\, M\times
  \set{0},
  \end{cases}
\end{equation}
where $M$ is a complete, $m$--dimensional, noncompact Riemannian
manifold with metric $g$, $\diver$ and $\nabla$ are respectively the
divergence and the gradient with respect to $g$, $p>1, q>\max\{p-1,
1\}$, the potential satisfies $V=V(x,t)>0$ a.e. in $M\times
(0,\infty)$ and the initial condition $u_0$ is nonnegative.

Local existence, finite time blow-up and global existence of
solutions to parabolic Cauchy problems have attracted much
attention in the literature. In particular, the following
semilinear parabolic Cauchy problem
\begin{equation}\label{ei1a}
\left\{
\begin{array}{ll}
 \,  \partial_t u - \Delta u\, =\,u^{q} \, &\textrm{in}\,\,\mathbb R^m\times
 (0,\infty)
\\&\\
\textrm{ }u \, = u_0& \textrm{in}\,\,  \mathbb R^m\times \{0\} \,,
\end{array}
\right.
\end{equation}
where $q>1, u_0\ge 0, u_0\in L^\infty(\mathbb R^m)$, has been
largely investigated. Indeed (see \cite{Fujita}, \cite{Fuj2} and
\cite{Lev}), problem \eqref{ei1a} does not admit global bounded
solutions for $1<q\le 1+\frac 2 m$. On the contrary, for $q>1+\frac
2 m$ global bounded solutions exist, provided that $u_0$ is
sufficiently small. For initial conditions $u_0\in L^p(\mathbb R^m)$
similar results have been obtained in the framework of mild
solutions in the space $C\big([0,T);L^p(\mathbb R^m)\big)$ in
\cite{Weiss}, \cite{Weiss2}.

Problem \eqref{Eq} with $(M, g)=(\mathbb R^m, g_{\text{flat}})$,
where $g_{\text{flat}}$ is the standard flat metric in the Euclidean
space, together with its generalization to a wider class of
operators of $p-$Laplace type or related to the porous medium
equation, has also been largely studied; without claim of
completeness we refer the reader to \cite{Galak1}, \cite{Galak2},
\cite{GalakLev}, \cite{MitPohozAbsence}, \cite{MitPohozMilan},
\cite{PoTe}, \cite{PuTe}, and references therein. In particular, in
\cite{MitPohozAbsence} it is shown that problem \eqref{Eq} with
$M=\mathbb R^m$ and $V\equiv 1$ does not admit nontrivial
nonnegative weak solutions, provided that
\[ p >\frac{2m}{m+1}\,, \quad q\leq p - 1 + \frac{p}{m}\,.\]

\medskip

Moreover, the blow-up result given in \cite{Fujita} has been
extended to the setting of Riemannian manifolds. To further describe
such results, let us introduce some notation. Let $(M, g)$ be a
complete noncompact Riemannian manifold, endowed with a smooth
Riemannian metric $g$. Fix any point $x_0\in M$, and for any $x\in
M$ denote by $r(x)=\textrm{dist}(x_0, x)$ the Riemannian distance
between $x_0$ and $x$. Moreover, let $B(x_0,r)$ be the geodesics
ball with center $x_0\in M$ and radius $r>0$, and let $\mu$ be the
Riemannian volume on $M$ with volume density $\sqrt g$.

In \cite{Zhang} it is proved that no nonnegative nontrivial weak
solutions to problem \eqref{Eq} with $p=2$ exist, provided there
exist $C>0, \alpha>2, \beta>-2$ such that, for all $r>0$ large
enough:
\begin{itemize}
\item [$(a)$] $\mu(B(x,r))\le C r^{\alpha}$ for all $x\in M$;
\item [$(b)$] $\frac{\partial \log \sqrt g}{\partial r}\le \frac C
r$; \item[$(c)$] $V=V(x)$, $V\in L^\infty_{loc}(M)$ and $C^{-1}
r(x)^{\beta}\leq V(x) \leq  C r(x)^{\beta}$;
\end{itemize}

Observe that if the Ricci curvature of $M$ is nonnnegative, then
$(a)-(b)$ are satisfied, see e.g. \cite{***}. On the other hand (see
Theorem 5.2.10 in \cite{Dav}, or Section 10.1 of \cite{Grig}),
hypotheses $(a)-(b)$ imply that $\lambda_1(M)=0,$ where
$\lambda_1(M)$ is the infimum of the $L^2-$ spectrum of the operator
$-\Delta\,$ on $M$\,.

\smallskip

The semilinear Cauchy problem
\begin{equation}\label{ei4}
\left\{
\begin{array}{ll}
 \,  \partial_t u = \Delta u\, +\, h(t) u^\nu \, &\textrm{in}\,\,\mathbb H^m\times (0,T)
\\&\\
\textrm{ }u \, = u_0& \textrm{in}\,\, \mathbb H^m\times \{0\} \,
\end{array}
\right.
\end{equation}
has been studied in \cite{BPT}, where $\mathbb H^m$ is the
$m-$dimensional hyperbolic space, $u_0$ is nonnegative and bounded
on $M$ and $h$ is a positive continuous function defined in
$[0,\infty)$; note that in this case we have $\lambda_1(\mathbb
H^N)=\frac{(N-1)^2}{4}.$

To be specific, it has been shown that if $h(t)\equiv 1\; (t\ge 0)$,
or if
\begin{equation}\label{e59}
\alpha_1 t^q \le h(t)\le \alpha_2 t^q \quad \textrm{for
any}\;\;t>t_0,
\end{equation}
for some $\alpha_1>0, \alpha_2>0, t_0>0$ and $q>-1$, then there
exist global bounded solutions for sufficiently small initial data
$u_0$. Moreover, when $h(t)=e^{\alpha t}\quad (t\ge 0)$ for some
$\alpha>0$, the authors showed that:
\begin{itemize}
\item[$(i)$] if $1< q < 1+\frac{\alpha}{\lambda_1(\mathbb H^m)},$
then every nontrivial bounded solution of problem \eqref{ei4}
blows up in finite time;

\item[$(ii)$] if $q >  1+\frac{\alpha}{\lambda_1(\mathbb H^m)},$
then problem \eqref{ei4} posses global bounded solutions for small
initial data ;

\item[$(iii)$] if $q = 1+\frac{\alpha}{\lambda_1(\mathbb H^m)}$
and $\alpha> \frac 2 3 \lambda_1(\mathbb H^m),$ then there exist
global bounded solutions of problem \eqref{ei4} for small initial
data.
\end{itemize}

%Let us mention that also in \cite{Shimo}, blow-up of radially
%symmetric solutions on $\mathbb H^N$ has been investigated.

\smallskip

Analogous results to those established in \cite{BPT} have been
obtained in \cite{P1}, for the problem
\begin{equation}\label{e610}
\left\{
\begin{array}{ll}
 \,  \partial_t u = \Delta u\, +\, h(t) u^q \, &\textrm{in}\,\,M\times (0,T)
\\&\\
\textrm{ }u \, = u_0& \textrm{in}\,\,  M\times \{0\} \,,
\end{array}
\right.
\end{equation}
where $M$ is a Cartan-Hadamard Riemannian manifold with sectional
curvature bounded above by a negative constant, and $u_0\in
L^\infty(M)$. Moreover, for initial conditions $u_0\in L^p(M)$
similar results have been established for mild solutions belonging
to $C\big([0,T);L^p(M)\big)$ in \cite{P2}.

\smallskip

Let us mention that nonexistence results of nonnegative nontrivial
solutions have been also much investigated for solutions to elliptic
equations and inequalities both on $\mathbb R^m$ (see, e.g.,
\cite{DaMit}, \cite{MitPohoz227}, \cite{MitPohoz359},
\cite{MitPohozMilan}, \cite{Mont}, \cite{DaLu}) and on Riemannian
manifolds (see \cite{GrigKond}, \cite{GrigS}, \cite{Kurta}
\cite{MMP1}, \cite{MasRigSet}, \cite{Sun1}, \cite{Sun2}). In
particular, the present paper is the natural continuation of
\cite{MMP1}, where some ideas and methods introduced in
\cite{GrigS}, \cite{GrigKond} and \cite{Kurta} have been developed.
Indeed, our results can be regarded as the parabolic counterpart of
those shown in \cite{MMP1}, concerning nonnegative weak solutions to
the inequality
\[ -\diver\pa{\abs{\nabla u}^{p-2}\nabla u}\geq V(x) u^q \quad \text{ in }\,
M\,.
\]
In \cite{MMP1}, as well as in \cite{GrigKond}, \cite{GrigS},
\cite{Sun1} and \cite{Sun2}, the key assumptions are concerned
with the parameters $p, q$ and the behavior of a suitable weighted
volume of geodesic balls, with density a negative power of the
potential $V(x)$.

As for the case of $\mathbb R^m$, also on Riemannian manifolds the
parabolic case presents substantial differences with respect to the
elliptic one. In fact, new test functions have to be used, and
suitable estimates of new integral terms are necessary. On the other
hand, as in the case of elliptic inequalities on Riemannian
manifolds, a simple adaptation of the methods used in $\mathbb R^m$
does not allow to obtain results as accurate as those we prove in
the present work. In the next two subsections we describe our main
results and some of their consequences; furthermore, we compare them
with results in the literature.

\subsection{Main results}
In order to formulate our main results, we shall introduce some
further notation and hypotheses. For each $R>0$, $\theta_1\geq 1$,
$\theta_2\geq 1$ let $S:=M\times[0,\infty)$ and
\[E_R:=\{(x,t)\in S\,:\, r(x)^{\theta_2}+t^{\theta_1}\leq R^{\theta_2} \}\,.\]
Let
\begin{equation*}
\begin{array}{rclrcl}
 \displaystyle\bar s_1\,&:=&\,\displaystyle\frac q{q-1}\theta_2\,,&\bar s_2\,&:=&\,\displaystyle\frac 1{q-1}\,,\\
 \displaystyle\bar s_3\,&:=&\,\displaystyle\frac{pq}{q-p+1}\theta_2\,,&\bar
s_4\,&:=&\,\displaystyle\frac{p-1}{q-p+1}\,.
\end{array}
\end{equation*}

The following conditions, that we call {\bf HP1} and {\bf HP2}, are
the main hypotheses under which we will derive our nonexistence
results for nonnegative nontrivial weak solutions of problem
\eqref{Eq}.

 \textbf{HP1.}\;\,Assume that: $(i)$\; there exist constants
$\theta_1\geq1$, $\theta_2\geq1$, $C_0>0$, $C>0$, $R_0>0$,
$\eps_0>0$ such that
 for every $R>R_0$ and for every $0<\eps<\eps_0$ one has
 \begin{equation}\label{hp1a}\int\int_{E_{2^{1/\theta_2}R}\setminus E_R} t^{(\theta_1-1)\left(\frac
 q{q-1}-\eps\right)}V^{-\frac{1}{q-1}+\eps}d\mu dt\leq C
 R^{\bar s_1+C_0\eps}(\log R)^{s_2}\,,
\end{equation}
for some $0\leq s_2<\bar s_2$\,;

$(ii)$\; for the same constants as above, for every $R>R_0$ and for
every $0<\eps<\eps_0$ one has
\begin{equation}\label{hp1b}\int\int_{E_{2^{1/\theta_2}R}\setminus E_R}
r(x)^{(\theta_2-1)p\left(\frac{q}{q-p+1}-\eps\right)}V^{-\frac{p-1}{q-p+1}+\eps}d\mu
dt\leq C R^{\bar s_3+C_0\eps}(\log R)^{s_4}\,,\end{equation} for
some $0\leq s_4<\bar s_4\,.$

\medskip

\textbf{HP2.}\;\,Assume that: $(i)$\; there exist constants
$\theta_1\geq1$, $\theta_2\geq1$, $C_0>0$, $C>0$, $R_0>0$,
$\eps_0>0$ such that
 for every $R>R_0$ and for every $0<\eps<\eps_0$ one has
\begin{eqnarray}
\label{hp2a}\int\int_{E_{2^{1/\theta_2}R}\setminus E_R}
t^{(\theta_1-1)\left(\frac
 q{q-1}-\eps\right)}V^{-\frac{1}{q-1}+\eps}d\mu dt&\leq& C
 R^{\bar s_1+C_0\eps}(\log R)^{\bar s_2}\,,\\
 \label{hp2aa} \int\int_{E_{2^{1/\theta_2}R}\setminus E_R} t^{(\theta_1-1)\left(\frac
 q{q-1}+\eps\right)}V^{-\frac{1}{q-1}-\eps}d\mu dt&\leq& C
 R^{\bar s_1+C_0\eps}(\log R)^{\bar s_2}\,;
\end{eqnarray}

$(ii)$\; for the same constants as above, for every $R>R_0$ and for
every $0<\eps<\eps_0$ one has
\begin{eqnarray}
\label{hp2b}\int\int_{E_{2^{1/\theta_2}R}\setminus E_R}
r(x)^{(\theta_2-1)p\left(\frac{q}{q-p+1}-\eps\right)}V^{-\frac{p-1}{q-p+1}+\eps}d\mu
dt& \leq &C R^{\bar s_3+C_0\eps}(\log R)^{\bar
s_4}\,,\\
\label{hp2bb}\int\int_{E_{2^{1/\theta_2}R}\setminus E_R}
r(x)^{(\theta_2-1)p\left(\frac{q}{q-p+1}+\eps\right)}V^{-\frac{p-1}{q-p+1}-\eps}d\mu
dt& \leq &C R^{\bar s_3+C_0\eps}(\log R)^{\bar s_4}\,.
\end{eqnarray}

%
%\textbf{HP3.}\;\,Assume that: $(i)$\; there exist $\theta_1\geq1$,
%$\theta_2\geq1$, $C_0>0$, $C>0$, $k_1>0, \lambda>0$,
%$\tau_1>\max\left\{1, \frac{q-1}{q}(k_1+1)\right\}$, $R_0>0$,
%$\eps_0>0$ such that
% for every $R>R_0$ and for every $0<\eps<\eps_0$ one has
% \begin{equation}\label{hp3a}\int\int_{E_{2^{1/\theta_2}R}\setminus E_R} t^{(\theta_1-1)\left(\frac
% q{q-1}-\eps\right)}V^{-\frac{1}{q-1}+\eps}d\mu dt\leq C
% R^{\bar s_1+C_0\eps}(\log R)^{k_1} e^{-\eps \lambda(\log
% R)^{\tau_1}}\,;
%\end{equation}
%
%
%$(ii)$\; there exist $k_2>0$,
%$\tau_2>\max\left\{1,\frac{q-p+1}{q}(k_2+1) \right\}$, for the
%same constants $\theta_1$, $\theta_2$, $C$, $C_0$, $R_0$,
%$\lambda$, $\eps_0>0$ as above, for every $R>R_0$ and for
%every $0<\eps<\eps_0$ one has
%\begin{equation}\label{hp3b}\int\int_{E_{2^{1/\theta_2}R}\setminus E_R}
%r(x)^{(\theta_2-1)p\left(\frac{q}{q-p+1}-\eps\right)}V^{-\frac{p-1}{q-p+1}+\eps}d\mu
%dt\leq C R^{\bar s_3+C_0\eps}(\log R)^{k_2} e^{-\eps
%\lambda(\log
% R)^{\tau_2}}\,.
% \end{equation}

\begin{rem}\label{rem1}
Passing to the limit as $\eps\to 0$ we see that, if {\bf HP1} holds,
then for the same constants as above conditions \eqref{hp1a} and
\eqref{hp1b} hold also for $\eps=0$. Similarly, if {\bf HP2} holds
then \eqref{hp2a} and \eqref{hp2b} (or equivalently \eqref{hp2aa}
and \eqref{hp2bb}) are satisfied also with $\eps=0$.
%while if one has condition {\bf HP3} then \eqref{hp3a}
%and \eqref{hp3b} hold also for $\eps=0$.
\end{rem}

\smallskip

We prove the following theorems (for the definition of weak solution
see Definition \ref{def1} below).
\begin{theorem}\label{thm1}
Let $p>1$, $q>\max\{p-1, 1\}$, $V>0$ a.e. in $M\times (0,
\infty)$, $V \in L^1_{loc}\pa{M\times [0, \infty)}$ and $u_0\in
L^1_{loc}(M)$, $u_0\geq 0$ a.e. in $M$. Let $u$ be a nonnegative
weak solution of problem \eqref{Eq}. Assume condition {\bf HP1}.
Then $u=0$ a.e. in $S$\,.
\end{theorem}

\begin{theorem}\label{thm2}
Let $p>1$, $q>\max\{p-1, 1\}$, $V>0$ a.e. in $M\times (0, \infty)$,
$V \in L^1_{loc}\pa{M\times [0, \infty)}$ and $u_0\in L^1_{loc}(M)$,
$u_0\geq 0$ a.e. in $M$. Let $u$ be a nonnegative weak solution of
problem \eqref{Eq}. Assume condition {\bf HP2}. Then $u=0$ a.e. in
$S$\,.
\end{theorem}
%
%\begin{theorem}\label{thm3}
%Let $p>1$, $q>q>\max\{p-1, 1\}$, $V>0$ a.e. in $M\times (0,
%\infty)$, $V \in L^1_{loc}\pa{M\times [0, \infty)}$ and $u_0\in
%L^1_{loc}(M)$, $u_0\geq 0$ a.e. in $M$. Let $u$ be a nonnegative
%weak solution of problem \eqref{Eq}. Assume condition {\bf HP3}.
%Then $u=0$ a.e. in $S$\,.
%\end{theorem}

We should note that, to the best of our knowledge, no nonexistence
results for linear or nonlinear parabolic equations on complete,
noncompact Riemannian manifolds have been obtained in the literature
under conditions similar to {\bf HP1} and {\bf HP2}, nor using the
techniques that we exploit to prove Theorems \ref{thm1} and
\ref{thm2}. Even if Theorems \ref{thm1} and \ref{thm2} can be
regarded as the natural parabolic counterparts of the results in
\cite{MMP1} for elliptic equations, their proofs are substantially
different from those in the elliptic case. Moreover, we should also
observe that in \cite{MMP1} a nonexistence result for the stationary
problem was obtained under a different assumption than the
stationary counterparts of the conditions {\bf HP1} and {\bf HP2}
introduced in the present work (see \cite[condition {\bf
HP3}]{MMP1}). An analogous result which could give rise to
nontrivial applications cannot be deduced using our methods for
parabolic equations, and the question whether a hypothesis
corresponding to \cite[condition {\bf HP3}]{MMP1} can be introduced
also in the parabolic setting in order to prove nonexistence results
still remains to be understood.
%
%
% In this respect, we should also observe that in \cite{MMP1}
%also some nonexistence results under a more general assumption than
%the stationary counterpart of {\bf HP1} and {\bf HP2} have been
%obtained. However, analogous results cannot be deduced using our
%methods for parabolic equations and remain to be understood.

\subsection{Applications} This subsection is devoted to the
discussion of some consequences of Theorems \ref{thm1} and
\ref{thm2} and to comparison with existing results in the
literature.

\begin{cor}\label{cor1}
Let $(M, g)=(\mathbb R^m, g_{\text{flat}})$, $V\equiv 1$, $p> 1$.
Suppose that
\begin{equation}\label{eq55}
\max\{1,p-1\}<q \leq \frac p m + p -1\,.
\end{equation}
Let $u$ be a nonnegative weak solution of problem \eqref{Eq}. Then
$u=0$ a.e. in $S$\,.
\end{cor}
Note that condition \eqref{eq55} in particular requires that
$p>\frac{2m}{m+1}$. Note also that Corollary \ref{cor1} agrees with
results in \cite{MitPohozAbsence}. Furthermore, for $p=2$ we recover
the results on the Laplace operator in \cite{Fujita, Haya}.

\smallskip

\begin{cor}\label{cor2}
Let $M$ be a complete noncompact Riemannian manifold, $p>1$,
$q>\max\{p-1, 1\}$ and $u_0\in L^1_{loc}(M)$, $u_0\geq 0$ a.e. in
$M$. Suppose the potential $V\in L^1_{loc}\pa{M\times [0, \infty)}$
satisfies
\begin{equation}\label{eq70}
V(x,t) \geq f(t) h(x) \quad \textrm{for a.e.}\;\; (x,t)\in S,
\end{equation}
where $f:(0,\infty)\to \mathbb R$, $h: M \to \mathbb R$ are two
functions satisfying
\begin{equation}\label{eq56}
0 < f(t) \leq C (1+t)^\alpha\,\,\,\textrm{for a.e. }t\in
(0,\infty)\qquad\textrm{and}\qquad0< h(x) \leq C (1+
r(x))^{\beta}\,\,\,\textrm{for a.e. }x\in M
\end{equation}
and
\begin{eqnarray}
\label{57}&&\int_0^T f(t)^{-\frac 1{q-1}}\,dt\leq C
T^{\sigma_2}(\log T)^{\delta_2}\,,\qquad \int_0^T
f(t)^{-\frac{p-1}{q-p+1}}\, dt\leq C T^{\sigma_4} (\log
T)^{\delta_4}\,,\\
\label{58}&&\int_{B_R} h(x)^{-\frac 1{q-1}}\,d\mu\leq C
R^{\sigma_1}(\log R)^{\delta_1}\,,\qquad\int_{B_R}
h(x)^{-\frac{p-1}{q-p+1}}\, d\mu\leq C R^{\sigma_3} (\log
R)^{\delta_3}
\end{eqnarray}
for $T,R$ large enough, with
$\alpha,\beta,\sigma_1,\sigma_2,\sigma_3,\sigma_4,\delta_1,\delta_2,\delta_3,\delta_4\geq0$
and $C>0$. Assume that
\begin{itemize}
 \item[i)] $\delta_1+\delta_2<\frac{1}{q-1}\,,\quad\delta_3+\delta_4<\frac{p-1}{q-p+1}\,$;
 \item[ii)]
 $0\leq\sigma_2\leq\frac{q}{q-1}\,,\quad0\leq\sigma_3\leq\frac{pq}{q-p+1}\,$;
 \item[iii)] if $\sigma_2=\frac{q}{q-1}$ then $\sigma_1=0\,$, if
 $\sigma_3=\frac{pq}{p-q+1}$ then $\sigma_4=0\,$;
 \item[iv)]
 $\sigma_1\sigma_4\leq\left(\frac{q}{q-1}-\sigma_2\right)\left(\frac{pq}{q-p+1}-\sigma_3\right)\,$.
\end{itemize}
Then problem \eqref{Eq} does not admit any nontrivial nonnegative
weak solution.
\end{cor}

\begin{cor}\label{cor3}
Let $M$ be a complete noncompact Riemannian manifold, $p>1$,
$q>\max\{p-1, 1\}$ and $u_0\in L^1_{loc}(M)$, $u_0\geq 0$ a.e. in
$M$. Assume that $V\in L^1_{loc}\pa{M\times [0, \infty)}$ satisfies
condition \eqref{eq70} with $f:(0,\infty)\to \mathbb R$, $h: M \to
\mathbb R$ such that
\begin{equation}\label{eq60}
\begin{array}{ll}
\displaystyle C^{-1} (1+t)^{-\alpha} \leq f(t) \leq C
(1+t)^\alpha&\textrm{for a.e. }t\in
(0,\infty)\,\\
\displaystyle C^{-1} (1+r(x))^{-\beta}\leq h(x) \leq C (1+
r(x))^{\beta}&\textrm{for a.e. }x\in M
\end{array}
\end{equation}
and \eqref{57}, \eqref{58} hold for $T,R$ sufficiently large,
$\alpha,\beta,\sigma_1,\sigma_2,\sigma_3,\sigma_4,\delta_1,\delta_2,\delta_3,\delta_4\geq0$
and $C>0$. Suppose that
\begin{itemize}
 \item[i)] $\delta_1+\delta_2\leq\frac{1}{q-1}\,,\quad\delta_3+\delta_4\leq\frac{p-1}{q-p+1}\,$;
 \item[ii)]
 $0\leq\sigma_2\leq\frac{q}{q-1}\,,\quad0\leq\sigma_3\leq\frac{pq}{q-p+1}\,$;
 \item[iii)] if $\sigma_2=\frac{q}{q-1}$ then $\sigma_1=0\,$, if
 $\sigma_3=\frac{pq}{p-q+1}$ then $\sigma_4=0\,$;
 \item[iv)]
 $\sigma_1\sigma_4\leq\left(\frac{q}{q-1}-\sigma_2\right)\left(\frac{pq}{q-p+1}-\sigma_3\right)\,$.
\end{itemize}
Then problem \eqref{Eq} does not admit any nontrivial nonnegative
weak solution.
\end{cor}

\begin{rem}\label{rem2}
\begin{itemize}
\item[i)] We explicitly note that the hypotheses in Corollaries \ref{cor2}
and \ref{cor3} allow for a potential $V$ that can also be
independent of $x\in M$ or of $t\in[0,\infty)$.

\item[ii)] In the particular case of the Laplace--Beltrami operator, i.e.
for $p=2$, from Corollaries \ref{cor2}, \ref{cor3} we have the
following results:

\emph{Let $V$ satisfy condition \eqref{eq70}, with $f:(0,\infty)\to
\mathbb R$, $h M \to \mathbb R$ such that \eqref{eq56} holds and
\begin{eqnarray}
\label{59}&&\int_{B_R}h(x)^{-\frac 1{q-1}}\,d\mu\leq C
R^{\sigma_1}(\log R)^{\delta_1}\,,\qquad \int_0^T
f(t)^{-\frac{1}{q-1}}\, dt\leq C T^{\sigma_2} (\log T)^{\delta_2}\,
\end{eqnarray}
for $T,R$ large enough, with
$\alpha,\beta,\sigma_1,\sigma_2,\delta_1,\delta_2\geq0$, $C>0$ and
\[\delta_1+\delta_2<\frac{1}{q-1}\,,\qquad\sigma_1+2\sigma_2\leq\frac{2q}{q-1}.\]
Then there exists no nonnegative, nontrivial weak solution of
problem \eqref{Eq} with $p=2$.}

\emph{Similarly, if condition \eqref{eq70} on $V$ holds with $f,h$
satisfying \eqref{eq60} and \eqref{59} for $T,R$ sufficiently large,
$\alpha,\beta,\sigma_1,\sigma_2,\delta_1,\delta_2\geq0$, $C>0$ and
if
\[\delta_1+\delta_2\leq\frac{1}{q-1}\,,\qquad\sigma_1+2\sigma_2\leq\frac{2q}{q-1}\,,\]
then there exists no nonnegative, nontrivial weak solution of
problem \eqref{Eq} with $p=2$.}
\end{itemize}
\end{rem}

We should note that, even if in view of Remark \ref{rem2}-i) problem
\eqref{ei4} on the hyperbolic space could in principle be addressed,
we cannot actually obtain nonexistence results for it using our
results. In fact, condition \eqref{58} is not satisfied if
$M=\mathbb H^m$ and $h\equiv 1$, due to the exponential volume
growth of geodesic balls in the hyperbolic space. Therefore, we do
not recover the results given in \cite{BPT} (see also \cite{P1}).
This is essentially due to the fact that in \cite{BPT} spectral
analysis and heat kernel estimates on $\mathbb H^m$ have been used.
Similar methods have also been used on Cartan-Hadamard manifolds in
\cite{P1}. Clearly, such tools are not at disposal on general
Riemannian manifolds, that are the object of our investigation. On
the other hand, our hypotheses {\bf HP1} and {\bf HP2} include a
large class of Riemannian manifolds for which results in \cite{BPT}
or in \cite{P1} cannot be applied. In particular, this includes the
case of Riemannian manifolds that satisfy $(a), (b), (c)$ above,
also treated in \cite{Zhang}.

\smallskip

In \cite{Zhang} quite different methods from ours have been
employed, but also porous medium type nonlinear operators have been
considered. However, we remark that in this work we introduce new
techniques in the setting of parabolic equations on  Riemannian
manifolds. We obtain completely new results in the case of the
p-Laplace operator, which improve on those already present in the
literature even in the particular case of semililinear equations
involving the Laplacian. Indeed, we obtain more general nonexistence
results than those in \cite{Zhang} (see Example \ref{exe1} below).

\medskip

The paper is organized as follows: in Section \ref{sec2} we prove
some preliminary results, that will be used in the proof of the
theorems and corollaries stated in the Introduction; Section
\ref{sec3} contains the proof of Theorems \ref{thm1} and \ref{thm2},
while Section \ref{sec4} is devoted to the proof of the Corollaries.

\section{Auxiliary results}\label{sec2}
We begin with
\begin{defi}\label{def1}
Let $p>1$, $q>\max\{p-1, 1\}$, $V>0$ a.e. in $M\times (0, \infty)$,
$V \in L^1_{loc}\pa{M\times [0, \infty)}$ and $u_0\in L^1_{loc}(M)$,
$u_0\geq 0$ a.e. in $M$. We say that $u \in
W^{1,p}_{loc}(M\times[0,\infty))\cap L^q_{loc}(M\times[0,\infty); V
d\mu dt)$ is a \emph{weak solution} of problem \eqref{Eq} if $u\geq
0$ a.e. in $M\times (0, \infty)$  and for every $\psi \in
W^{1,p}(M\times[0,\infty))$, with $\psi\geq 0$ a.e. in $M\times [0,
\infty)$ and compact support, one has
\begin{equation}\label{Eq_weakSol}
  \intst{\psi u^q V} \leq \intst{\abs{\nabla u}^{p-2}\pair{\nabla u, \nabla \psi}} - \intst{u\,\partial_t\psi}-\int_M u_0 \psi(x, 0)\,d\mu.
\end{equation}
\end{defi}

The next lemmas will be the crucial tools we will use in the proof
of Theorems \ref{thm1} and \ref{thm2}.
\begin{lemma}
  Let $s \geq \max\set{1, \frac{q}{q-1}, \frac{pq}{q-p+1}}$ be fixed. Then there exists a constant $C>0$ such that for every $\alpha \in \frac{1}{2}\pa{-\min\set{1, p-1}, 0}$,
  every nonnegative weak solution $u$ of problem \eqref{Eq} and every $\varphi \in \operatorname{Lip}\pa{M\times [0, \infty)}$ with compact support and $0\leq \varphi \leq 1$ one has
  \begin{align}
  \label{1}&\frac{1}{2}  \intst{V u^{q+\alpha} \varphi^s} +\frac{3}{4} |\alpha|\intst{|\nabla u|^pu^{\alpha-1}\varphi^s}\\
  \nonumber&\hspace{1cm}\leq C\set{\abs{\alpha}^{-\frac{(p-1)q}{q-p+1}} \intst{\abs{\nabla \varphi}^{\frac{p\pa{q+\alpha}}{q-p+1}}V^{-\frac{p+\alpha-1}{q-p+1}}}
   + \intst{\abs{\partial_t\varphi}^{\frac{q+\alpha}{q-1}} V^{-\frac{\alpha+1}{q-1}}}}.
  \end{align}
\end{lemma}
\begin{proof}
  For any $\eps>0$ let $u_\eps := u+\eps$. Define $\psi = u_\eps^\alpha \varphi^s$; then $\psi$ is an admissible test function for problem \eqref{Eq}, with
  \begin{equation*}
    \nabla\psi = \alpha u_\eps^{\alpha-1} \varphi^s \nabla u + s\varphi^{s-1}u_\eps^\alpha\nabla\varphi, \quad \partial_t\psi=\alpha u_\eps^{\alpha-1}\varphi^s\partial_t u + s\varphi^{s-1}u_\eps^\alpha\partial_t\varphi.
  \end{equation*}
Inequality \eqref{Eq_weakSol} gives
\begin{equation}\label{EQ2}
  \intst{u^q u_\eps^\alpha \varphi^s V} \leq \alpha\intst{\abs{\nabla u}^pu_\eps^{\alpha-1}\varphi^s} +s\intst{\abs{\nabla u}^{p-2}\pair{\nabla u, \nabla\varphi}u_\eps^\alpha \varphi^{s-1}} + I,
\end{equation}
where
\begin{equation}
  I = -\alpha \intst{u_\eps^{\alpha-1}\varphi^s u \partial_t u} - s\intst{u u_\eps^{\alpha}  \varphi^{s-1} \partial_t\varphi } - \int_M{u_0 \pa{u_0+\eps}^\alpha \varphi^s\pa{x, 0}}\,d\mu.
\end{equation}
Now we have
\begin{align*}
-\alpha \intst{u_\eps^{\alpha-1}\varphi^s u \partial_t u} &= -\alpha\intst{u_\eps^{\alpha}\varphi^s  \partial_t u}-\alpha\eps\intst{u_\eps^{\alpha-1}\varphi^s \partial_t u}\\\\
&=-\frac{\alpha}{\alpha+1}\intst{\partial_t\big(u_\eps^{\alpha+1}\big)\varphi^s}+\eps\intst{\partial_t\big(u_\eps^{\alpha}\big)\varphi^s}\,.
\end{align*}
Since $u_\eps^\alpha,u_\eps^{\alpha+1}\in
W^{1,p}_{loc}(M\times[0,\infty))$ with $p>1$ and since
$\varphi^s\in W^{1,p'}(M\times[0,\infty))$ and has compact
support, integrating by parts we obtain
\begin{align*}
-\alpha \intst{u_\eps^{\alpha-1}\varphi^s u \partial_t u} =&
\frac{\alpha s}{\alpha+1}\intst{\varphi^{s-1} u_\eps^{\alpha+1}
\partial_t \varphi} -\eps s\intst{u_\eps^\alpha
\varphi^{s-1}\partial_t\varphi}\\
&+\frac{\alpha}{\alpha+1}\int_M{\varphi^s\pa{x,
0}\pa{u_0+\eps}^{\alpha+1}}\,d\mu-\eps\int_M{\varphi^s\pa{x,
0}\pa{u_0+\eps}^{\alpha}}\,d\mu,
\end{align*}
thus, recalling that $u_\eps=u+\eps$, we have
\begin{equation}
  I = -\frac{s}{\alpha+1} \intst{u_\eps^{\alpha+1}\varphi^{s-1}\partial_t\varphi}-\frac{1}{\alpha+1}\int_M{\pa{u_0+\eps}^{\alpha+1}\varphi^s\pa{x, 0}}\,d\mu.
\end{equation}
%Integrating by parts we have
%\begin{align*}
%&-\alpha \intst{u_\eps^{\alpha-1}\varphi^s u \partial_t u}\\
%&\qquad\qquad= \intst{\varphi^s u_\eps^\alpha \partial_tu} +
%s\intst{u u_\eps^\alpha
%\varphi^{s-1}\partial_t\varphi}+\int_M{\varphi^s\pa{x,
%0}u_0\pa{u_0+\eps}^\alpha}\,d\mu,
%\end{align*}
%thus
%\begin{equation}
%  I = -\frac{s}{\alpha+1} \intst{u_\eps^{\alpha+1}\varphi^{s-1}\partial_t\varphi}-\frac{1}{\alpha+1}\int_M{\pa{u_0+\eps}^{\alpha+1}\varphi^s\pa{x, 0}}\,d\mu.
%\end{equation}
This, combined with \eqref{EQ2}, yields
\begin{align}
  \intst{u^qu_\eps^\alpha \varphi^s V} &\leq \alpha\intst{\abs{\nabla u}^pu_\eps^{\alpha-1}\varphi^s} +s\intst{\abs{\nabla u}^{p-2}\pair{\nabla u, \nabla\varphi}u_\eps^\alpha \varphi^{s-1}} \\ \nonumber &-\frac{s}{\alpha+1} \intst{u_\eps^{\alpha+1}\varphi^{s-1}\partial_t\varphi}-\frac{1}{\alpha+1}\int_M{\pa{u_0+\eps}^{\alpha+1}\varphi^s\pa{x, 0}}\,d\mu
\end{align}
and then
\begin{align}\label{Eq_3}
&\abs{\alpha}\intst{\abs{\nabla
u}^pu_\eps^{\alpha-1}\varphi^s}+\intst{u^qu_\eps^\alpha \varphi^s V}
+\frac{1}{\alpha+1}\int_M{\pa{u_0+\eps}^{\alpha+1}\varphi^s\pa{x,
0}}\,d\mu\\\nonumber&\leq  s\intst{\abs{\nabla u}^{p-2}\pair{\nabla
u, \nabla\varphi}u_\eps^\alpha \varphi^{s-1}} -\frac{s}{\alpha+1}
\intst{u_\eps^{\alpha+1}\varphi^{s-1}\partial_t\varphi}.
\end{align}
Now we estimate the first integral in the right-hand side of \eqref{Eq_3} using Young's inequality, obtaining
\begin{align*}
   &s\intst {\varphi^{s-1}u^{\alpha}_{\eps}\abs{\nabla u}^{p-2}\pair{\nabla u, \nabla\varphi}}\\
   &\qquad\leq s\intst {\varphi^{s-1}u^{\alpha}_{\eps}\abs{\nabla u}^{p-1} \abs{\nabla\varphi}} \\
   &\qquad= \intst{\pa{\abs{\alpha}^{\frac{p-1}{p}}\varphi^{s\frac{p-1}{p}}u_\eps^{-\pa{\abs{\alpha}+1}\frac{p-1}{p}}\abs{\nabla
u}^{p-1} }\pa{s
\abs{\alpha}^{-\frac{p-1}{p}}\varphi^{\frac{s}{p}-1}u_\eps^{1-\frac{\abs{\alpha}+1}{p}}\abs{\nabla
\varphi}}} \\
&\qquad\leq\frac{\abs{\alpha}}{4}\intst{\varphi^su^{\alpha-1}_{\eps}\abs{\nabla
u}^p}+\frac{s}{p}\sq{\frac{4s(p-1)}{\abs{\alpha}p}}^{p-1}\intst{\varphi^{s-p}u^{p-(\abs{\alpha}+1)}_\eps\abs{\nabla\varphi}^p}.
\end{align*}
From \eqref{Eq_3} we deduce
\begin{align}\label{Eq_4}
&\frac34\abs{\alpha}\intst{\abs{\nabla
u}^pu_\eps^{\alpha-1}\varphi^s}+\intst{u^qu_\eps^\alpha \varphi^s V}
+\frac{1}{\alpha+1}\int_M{\pa{u_0+\eps}^{\alpha+1}\varphi^s\pa{x,
0}}\,d\mu\\
\nonumber&\qquad\leq
\frac{s}{p}\sq{\frac{4s(p-1)}{\abs{\alpha}p}}^{p-1}\intst{\varphi^{s-p}u^{p-(\abs{\alpha}+1)}_\eps\abs{\nabla\varphi}^p}
+\frac{s}{\alpha+1}
\intst{u_\eps^{\alpha+1}\varphi^{s-1}\abs{\partial_t\varphi}}.
\end{align}
Note that, by Young's inequality,
\begin{align*}
&\frac{s}{p}\sq{\frac{4s(p-1)}{\abs{\alpha}p}}^{p-1}\intst{\varphi^{s-p}u^{p-(\abs{\alpha}+1)}_\eps\abs{\nabla\varphi}^p}
 \\ &\qquad=\frac{s}{p}\sq{\frac{4s(p-1)}{\abs{\alpha}p}}^{p-1}
\intst{\pa{u_\eps^{p+\alpha-1}\varphi^{s\pa{\frac{p+\alpha-1}{q+\alpha}}}V^{\frac{p+\alpha-1}{q+\alpha}}}\pa{\abs{\nabla\varphi}^p\varphi^{s-p-s\pa{\frac{p+\alpha-1}{q+\alpha}}}V^{-\frac{p+\alpha-1}{q+\alpha}}}}
\\ &\qquad\leq \frac14 \intst{u_\eps^{q+\alpha}\varphi^s V} + C\pa{\alpha,
s}
\intst{\abs{\nabla\varphi}^{\frac{p\pa{q+\alpha}}{q-p+1}}\varphi^{s-p\pa{\frac{q+\alpha}{q-p+1}}}V^{-\frac{p+\alpha-1}{q+p-1}}}
\end{align*}
and
\begin{align*}
&\frac{s}{\alpha+1}\intst{u_\eps^{\alpha+1}\varphi^{s-1}\abs{\partial_t\varphi}}
\\
&\qquad=\frac{s}{\alpha+1}\intst{\pa{u_\eps^{\alpha+1}\varphi^{s\pa{\frac{\alpha+1}{q+\alpha}}}V^{\frac{\alpha+1}{q+\alpha}}}\pa{\varphi^{-s\pa{\frac{\alpha+1}{q+\alpha}}+s-1}\abs{\partial_t\varphi}V^{-\frac{\alpha+1}{q+\alpha}}}}
  \\&\qquad\leq \frac14\intst{u_\eps^{q+\alpha}\varphi^s V} + D\pa{\alpha, s}\intst{\abs{\partial_t\varphi}^{\frac{q+\alpha}{q-1}}\varphi^{s-\frac{q+\alpha}{q-1}} V^{-\frac{\alpha+1}{q-1}} }
\end{align*}
where
\[
C\pa{\alpha, s}= \frac{s}{p}\sq{\frac{4s(p-1)}{\abs{\alpha}p}}^{p-1}
\frac{q-p+1}{q+\alpha}\sq{\frac{\pa{q+\alpha}p}{4s\pa{p+\alpha-1}}\pa{\frac{4s\pa{p-1}}{p\abs{\alpha}}}^{1-p}}^{-\frac{p+\alpha-1}{q-p+1}}
\]
and
\[
D\pa{\alpha, s}=\frac{s}{\alpha+1}\frac{q-1}{q+\alpha}\pa{\frac{4s}{q+\alpha}}^{\frac{\alpha+1}{q-1}}.
\]
Substituting in \eqref{Eq_4} we have
\begin{align*}
  &\frac34\abs{\alpha}\intst{\abs{\nabla u}^pu_\eps^{\alpha-1}\varphi^s}+\intst{u^qu_\eps^\alpha \varphi^s V}\\&\quad-\frac12\intst{u_\eps^{q+\alpha}\varphi^s
  V}+\frac{1}{\alpha+1}\int_M{\pa{u_0+\eps}^{\alpha+1}\varphi^s\pa{x,
0}}\,d\mu\\&\qquad\leq  C\pa{\alpha,
s}\intst{\abs{\nabla\varphi}^{\frac{p\pa{q+\alpha}}{q-p+1}}
  \varphi^{s-p\pa{\frac{q+\alpha}{q-p+1}}}V^{-\frac{p+\alpha-1}{q-p+1}}}+D\pa{\alpha, s}\intst{\abs{\partial_t\varphi}^{\frac{q+\alpha}{q-1}}\varphi^{s-\frac{q+\alpha}{q-1}}V^{-\frac{\alpha+1}{q-1}}}.
\end{align*}
Now letting $\eps\ra0$ and applying Fatou's lemma, we get
\begin{align}\label{Eq_5}
 \frac{3}{4} |\alpha|\intst{|\nabla u|^pu^{\alpha-1}\varphi^s}&+\frac{1}{2}\intst{V u^{q+\alpha}\varphi^s}\\
 &\nonumber\quad\leq C\pa{\alpha, s}\intst{\abs{\nabla\varphi}^{\frac{p\pa{q+\alpha}}{q-p+1}}\varphi^{s-p\pa{\frac{q+\alpha}{q-p+1}}}V^{-\frac{p+\alpha-1}{q-p+1}}}\\
 &\nonumber\qquad+D\pa{\alpha,
 s}\intst{\abs{\partial_t\varphi}^{\frac{q+\alpha}{q-1}}\varphi^{s-\frac{q+\alpha}{q-1}}V^{-\frac{\alpha+1}{q-1}}}\,,
\end{align}
where we use the convention $|\nabla u|^pu^{\alpha-1}\equiv0$ on the
set where $u=0$, since $\nabla u=0$ a.e. on level sets of $u$. Now
since there exists a positive constant $C$, depending on $s,p,q$,
such that
\[
C\pa{\alpha, s} \leq C \abs{\alpha}^{-\frac{(p-1)q}{q-p+1}}, \quad
D\pa{\alpha, s} \leq C,
\]
and since $0\leq\varphi\leq1$ on $M\times[0,\infty)$, by our
assumptions on $s$ the conclusion follows from \eqref{Eq_5}.

\end{proof}

\begin{lemma}\label{LE_tech2}
Let $s\geq \max\set{1, \frac{q+1}{q-1}, \frac{2 pq}{q-p+1}}$ be
fixed. Then there exists a constant $C>0$ such that for every
nonnegative weak solution $u$ of equation \eqref{Eq}, every function
$\varphi \in \operatorname{Lip}(S)$ with compact support and $0 \leq
\varphi \leq 1$ and every $\alpha\in
\pa{-\frac{1}{2}\min\set{1,p-1,q-1,\frac{q-p+1}{p-1}}, 0}$ one has
%%% preferiamo chiedere magari  $0<t<\min\set{1, \frac{q-p+1}{p-1},p-1}$ e $s \geq\frac{pq}{q-(t+1)(p-1)}$ invece? perÃ² cosÃ¬ s dipende da t...
\begin{align}
  \label{2.10}&\intst{\varphi^su^q V}\\
  \nonumber&\quad\leq C
     \pa{\abs{\alpha}^{-1-\frac{(p-1)q}{(q-p+1)}}\intst{{V^{-\frac{p+\alpha-1}{q-p+1}}\abs{\nabla\varphi}^{\frac{p(q+\alpha)}{q-p+1}}}}+|\alpha|^{-1} \intst{\abs{\partial_t\varphi}^{\frac{q+\alpha}{q-1}} V^{-\frac{\alpha+1}{q-1}}} }^\frac{p-1}{p}\\
  \nonumber&\qquad\times\pa{\int\int_{S\setminus K} V^{-\frac{(1-\alpha)(p-1)}{q-(1-\alpha)(p-1)}}
     \abs{\nabla\varphi}^\frac{pq}{q-(1-\alpha)(p-1)}\,d\mu dt}^\frac{q-(1-\alpha)(p-1)}{pq}\pa{\int\int_{S\setminus K}\varphi^su^q V\,d\mu dt}^\frac{(1-\alpha)(p-1)}{pq} \\
  \nonumber&\qquad+C\left(\int\int_{S\setminus K}\varphi^s u^{q+\alpha}V\,d\mu dt\right)^\frac{1}{q+\alpha}\pa{\intst{V^{-\frac{1}{q+\alpha-1}}
     \abs{\partial_t\varphi}^{\frac{q+\alpha}{q+\alpha-1}}}}^{\frac{q+\alpha-1}{q+\alpha}}.
\end{align}
with $K=\{\pa{x, t}\in S : \varphi(x, t)=1\}$.
\end{lemma}
%%%***Remark per noi: sotto le ipotesi su $q$ vale $s>p$***

\begin{proof}
Under our assumptions $\psi=\varphi^s$ is a feasible test function
in equation \eqref{Eq_weakSol}. Thus we obtain
\begin{equation}\label{6}
 \intst{ \varphi^s u^q V}\leq s \intst{\varphi^{s-1}\abs{\nabla u}^{p-2}\pair{\nabla
 u,\nabla\varphi}}-s\intst{u\varphi^{s-1}\partial_t \varphi}
 -\int_M u_0(x)\varphi^s(x,0)\,d\mu\,.
\end{equation}
Through an application of H\"{o}lder's inequality we obtain
\begin{align}
  \label{7}\intst{u\varphi^{s-1}|\partial_t \varphi|} \leq
  \left(\int\int_{S\setminus K} u^{q+\alpha}V \varphi^s\, d\mu dt
  \right)^{\frac1{q+\alpha}}\left(\intst{V^{-\frac 1{q+\alpha-1}}\varphi^{\frac{(s-1)(q+\alpha)-s}{q+\alpha-1}}|\partial_t
  \varphi|^{\frac{q+\alpha}{q+\alpha-1}}}\right)^{\frac{q+\alpha-1}{q+\alpha}}\,.
\end{align}
On the other hand, using again H\"{o}lder's inequality we obtain
\begin{align}
  \label{7a}&\intst{ s\varphi^{s-1}\abs{\nabla u}^{p-1}\abs{\nabla\varphi}}\\
  \nonumber&\qquad=s\intst{\pa{\varphi^{\frac{p-1}{p}s} \abs{\nabla u}^{p-1}u^{-\frac{p-1}{p}(1-\alpha)}}
            \pa{\varphi^{\frac{s}{p}-1}u^{\frac{p-1}{p}(1-\alpha)}\abs{\nabla\varphi}}}\\
  \nonumber&\qquad\leq s\pa{\intst{\varphi^s \abs{\nabla u}^pu^{\alpha-1}}}^\frac{p-1}{p}
            \pa{\intst{\varphi^{s-p}u^{(p-1)(1-\alpha)}\abs{\nabla\varphi}^p}}^\frac{1}{p}.
\end{align}
Moreover from equation \eqref{1} we deduce
\begin{align}\label{2.6}
 \intst{{\varphi^s\abs{\nabla u}^pu^{\alpha-1}}}& \leq C
  |\alpha|^{-1-\frac{(p-1)q}{q-p+1}}\intst{{V^{-\frac{p+\alpha-1}{q-p+1}}\abs{\nabla\varphi}^{\frac{p(q+\alpha)}{q-p+1}}}}\\
  \nonumber&+
  C|\alpha|^{-1}\intst{\abs{\partial_t\varphi}^{\frac{q+\alpha}{q-1}} V^{-\frac{\alpha+1}{q-1}}},
\end{align}
with $C>0$ depending on $s$. Thus from \eqref{6}, \eqref{7},
\eqref{7a} and \eqref{2.6} we obtain
\begin{align}
\label{2.8}&  \intst{\varphi^su^q V} \\
\nonumber& \leq C\left\{
  |\alpha|^{-1-\frac{(p-1)q}{q-p+1}}\intst{{V^{-\frac{p+\alpha-1}{q-p+1}}\abs{\nabla\varphi}^{\frac{p(q+\alpha)}{q-p+1}}}}
  +|\alpha|^{-1} \intst{\abs{\partial_t\varphi}^{\frac{q+\alpha}{q-1}} V^{-\frac{\alpha+1}{q-1}}} \right\}^\frac{p-1}{p}\\
\nonumber &\times\pa{\intst{ \varphi^{s-p}
  u^{(p-1)(1-\alpha)}\abs{\nabla\varphi}^p}}^\frac{1}{p}\\
\nonumber & +C\left(\int\int_{S\setminus K} u^{q+\alpha}V
\varphi^s\, d\mu dt
  \right)^{\frac1{q+\alpha}}\left(\intst{V^{-\frac 1{q+\alpha-1}}\varphi^{\frac{(s-1)(q+\alpha)-s}{q+\alpha-1}}|\partial_t
  \varphi|^{\frac{q+\alpha}{q+\alpha-1}}}\right)^{\frac{q+\alpha-1}{q+\alpha}}\,.
\end{align}
We use again H\"{o}lder's inequality with exponents
\[
a=\frac{q}{(1-\alpha)(p-1)}, \qquad
b=\frac{a}{a-1}=\frac{q}{q-(1-\alpha)(p-1)}
\]
to obtain
\begin{align*}
 &\intst{\varphi^{s-p}u^{(p-1)(1-\alpha)}\abs{\nabla\varphi}^p}\\
 &\qquad\leq \pa{\int\int_{S\setminus K}\varphi^su^q
     V\,d\mu dt}^\frac{(1-\alpha)(p-1)}{q}\pa{\int\int_{S\setminus
     K}\varphi^{s-\frac{pq}{q-(1-\alpha)(p-1)}} V^{-\frac{(1-\alpha)(p-1)}{q-(1-\alpha)(p-1)}}
     \abs{\nabla\varphi}^\frac{pq}{q-(1-\alpha)(p-1)}\,d\mu dt}^\frac{q-(1-\alpha)(p-1)}{q}.
\end{align*}
Substituting into \eqref{2.8} we have
\begin{align}
\label{2.8a}&  \intst{\varphi^su^q V} \\
\nonumber& \leq C\left\{
  |\alpha|^{-1-\frac{(p-1)q}{q-p+1}}\intst{{V^{-\frac{p+\alpha-1}{q-p+1}}\abs{\nabla\varphi}^{\frac{p(q+\alpha)}{q-p+1}}}}
  +|\alpha|^{-1} \intst{\abs{\partial_t\varphi}^{\frac{q+\alpha}{q-1}} V^{-\frac{\alpha+1}{q-1}}} \right\}^\frac{p-1}{p}\\
\nonumber &\times\pa{\int\int_{S\setminus K}\varphi^su^q
     V\,d\mu dt}^\frac{(1-\alpha)(p-1)}{qp}\pa{\int\int_{S\setminus
     K}\varphi^{s-\frac{pq}{q-(1-\alpha)(p-1)}} V^{-\frac{(1-\alpha)(p-1)}{q-(1-\alpha)(p-1)}}
     \abs{\nabla\varphi}^\frac{pq}{q-(1-\alpha)(p-1)}\,d\mu dt}^\frac{q-(1-\alpha)(p-1)}{qp}\\
\nonumber & +C\left(\int\int_{S\setminus K} u^{q+\alpha}V
\varphi^s\, d\mu dt
  \right)^{\frac1{q+\alpha}}\left(\intst{V^{-\frac 1{q+\alpha-1}}\varphi^{\frac{(s-1)(q+\alpha)-s}{q+\alpha-1}}|\partial_t
  \varphi|^{\frac{q+\alpha}{q+\alpha-1}}}\right)^{\frac{q+\alpha-1}{q+\alpha}}\,.
\end{align}
Now inequality \eqref{2.10} immediately follows from the previous
relation, by our assumptions on $s,\alpha$ and since
$0\leq\varphi\leq1$.

\end{proof}

\begin{cor}
Under the hypotheses of Lemma \ref{LE_tech2} one has
\begin{align}
\label{2.11}&  \intst{\varphi^su^q V} \\
\nonumber& \leq C\left\{
  |\alpha|^{-1-\frac{(p-1)q}{q-p+1}}\intst{{V^{-\frac{p+\alpha-1}{q-p+1}}\abs{\nabla\varphi}^{\frac{p(q+\alpha)}{q-p+1}}}}
  +|\alpha|^{-1} \intst{\abs{\partial_t\varphi}^{\frac{q+\alpha}{q-1}} V^{-\frac{\alpha+1}{q-1}}} \right\}^\frac{p-1}{p}\\
\nonumber &\times\pa{\int\int_{S\setminus K}\varphi^su^q
     V\,d\mu dt}^\frac{(1-\alpha)(p-1)}{qp}\pa{\int\int_{S\setminus
     K}V^{-\frac{(1-\alpha)(p-1)}{q-(1-\alpha)(p-1)}}
     \abs{\nabla\varphi}^\frac{pq}{q-(1-\alpha)(p-1)}\,d\mu dt}^\frac{q-(1-\alpha)(p-1)}{qp}\\
\nonumber & +C\pa{\abs{\alpha}^{-\frac{(p-1)q}{q-p+1}}
\intst{\abs{\nabla\varphi}^{\frac{p\pa{q+\alpha}}{q-p+1}}
     V^{-\frac{p+\alpha-1}{q-p+1}}} + \intst{\abs{\partial_t\varphi}^{\frac{q+\alpha}{q-1}}
     V^{-\frac{\alpha+1}{q-1}}}}^{\frac{1}{q+\alpha}}\\
\nonumber &\times\left(\intst{V^{-\frac 1{q+\alpha-1}}|\partial_t
  \varphi|^{\frac{q+\alpha}{q+\alpha-1}}}\right)^{\frac{q+\alpha-1}{q+\alpha}}\,.
\end{align}
\end{cor}

\begin{proof}
The conclusion immediately follows combining \eqref{2.10} and
\eqref{1}.
\end{proof}

\begin{lemma}\label{LE_tech3}
Let $s\geq \max\set{1, \frac{q+1}{q-1}, \frac{2 pq}{q-p+1}}$ be
fixed. Then there exists a constant $C>0$ such that for every
nonnegative weak solution $u$ of equation \eqref{Eq}, every function
$\varphi \in \operatorname{Lip}(S)$ with compact support and $0 \leq
\varphi \leq 1$ and every $\alpha\in
\pa{-\frac{1}{2}\min\set{1,p-1,q-1,\frac{q-p+1}{p-1}}, 0}$ one has
%%% preferiamo chiedere magari  $0<t<\min\set{1, \frac{q-p+1}{p-1},p-1}$ e $s \geq\frac{pq}{q-(t+1)(p-1)}$ invece? perÃ² cosÃ¬ s dipende da t...
\begin{align}
  \label{eq47}&\intst{\varphi^su^q V}\\
  \nonumber&\quad\leq C
     \pa{\abs{\alpha}^{-1-\frac{(p-1)q}{(q-p+1)}}\intst{{V^{-\frac{p+\alpha-1}{q-p+1}}\abs{\nabla\varphi}^{\frac{p(q+\alpha)}{q-p+1}}}}+
     |\alpha|^{-1} \intst{\abs{\partial_t\varphi}^{\frac{q+\alpha}{q-1}} V^{-\frac{\alpha+1}{q-1}}} }^\frac{p-1}{p}\\
  \nonumber&\qquad\times\pa{\int\int_{S\setminus K} V^{-\frac{(1-\alpha)(p-1)}{q-(1-\alpha)(p-1)}}
     \abs{\nabla\varphi}^\frac{pq}{q-(1-\alpha)(p-1)}\,d\mu dt}^\frac{q-(1-\alpha)(p-1)}{pq}\pa{\int\int_{S\setminus K}\varphi^su^q V\,d\mu dt}^\frac{(1-\alpha)(p-1)}{pq} \\
  \nonumber&\qquad+C\left(\int\int_{S\setminus K} \varphi^su^{q}V\,d\mu dt\right)^\frac{1}{q}\pa{\intst{V^{-\frac{1}{q-1}}
     \abs{\partial_t\varphi}^{\frac{q}{q-1}}}}^{\frac{q-1}{q}}.
\end{align}
with $S=M\times [0, \infty)$ and $K=\{\pa{x, t}\in S : \varphi(x,
t)=1\}$.
\end{lemma}

\begin{proof}
Inequality \eqref{eq47} can be proved in the same way as
\eqref{2.10}, where the only difference with respect to the above
argument is that in this case one has to use inequality \eqref{7}
with $\alpha=0$.
\end{proof}

\section{Proof of Theorems \ref{thm1} and \ref{thm2}}\label{sec3}
\begin{proof}[\it Proof of Theorem \ref{thm1}]
For any fixed $R>0$ sufficiently large, let $\alpha:=-\frac 1{\log
R}$. Fix any $C_1 > \frac{C_0 + \theta_2+1}{\theta_2}$ with $C_0$
and $\theta_2$ as in {\bf HP1}. Define for all $(x,t)\in S$
\begin{equation}\label{cutoff}
\varphi(x,t):=\left\{
\begin{array}{ll}
\,   1 &\textrm{if}\,\,(x,t)\in E_R\,,
\\& \\
\left(\frac{r(x)^{\theta_2}+t^{\theta_1}}{R^{\theta_2}}\right)^{C_1\alpha}&
\textrm{if\ \ } (x,t)\in E_R^c\,,
\end{array}
\right.
\end{equation}
and for all $n\in \mathbb N$
\begin{equation}\label{cutoff2}
\eta_n(x,t):=\left\{
\begin{array}{ll}
\,   1 &\textrm{if\ \ }  (x,t)\in E_{nR}\,,
\\& \\
2- \frac{r(x)^{\theta_2}+t^{\theta_1}}{(nR)^{\theta_2}}    & \textrm{if\ \ } (x,t)\in E_{2^{1/\theta_2}nR}\setminus E_{nR}\,,\\& \\
0 & \textrm{if\ \ } (x,t)\in E_{2^{1/\theta_2}nR}^c\,.
\end{array}
\right.
\end{equation}
Let
\begin{equation}\label{eq13}
\varphi_n(x,t):=\eta_n(x,t)\varphi(x,t)\quad \textrm{for all}\;\, (x,t)\in S\,.
\end{equation}
We have $\varphi_n\in \operatorname{Lip}(S)$ with $0\leq
\varphi_n\leq 1$; furthermore,
\begin{align*}
\partial_t \varphi_n= \eta_n\partial_t\varphi + \varphi\partial_t
\eta_n,\qquad\qquad\nabla \varphi_n= \eta_n\nabla\varphi +
\varphi\nabla \eta_n
\end{align*}
a.e. in $S$, and for every $a\geq 1$
\begin{align*}
|\partial_t \varphi_n|^a &\leq 2^{a-1}(|\partial_t  \varphi|^a +
\varphi^a|\partial_t  \eta_n|^a),&|\nabla\varphi_n|^a &\leq
2^{a-1}(|\nabla \varphi|^a + \varphi^a|\nabla \eta_n|^a)
\end{align*}
a.e. in $S$. Now we use $\varphi_n$ in formula \eqref{1}, with any
fixed $s \geq \max\set{1, \frac{q}{q-1}, \frac{pq}{q-p+1}}$, and we
see that for some positive constant $C$ and for every $n\in \mathbb
N$ and every small enough $|\alpha|>0$, we have
\begin{align}\label{eq12}
  & \intst{V u^{q+\alpha} \varphi_n^s} \\
  \nonumber&\quad\leq C\set{\abs{\alpha}^{-\frac{(p-1)q}{q-p+1}} \intst{\abs{\nabla
  \varphi_n}^{\frac{p\pa{q+\alpha}}{q-p+1}}V^{-\frac{p+\alpha-1}{q-p+1}}}
   + \intst{\abs{\partial_t\varphi_n}^{\frac{q+\alpha}{q-1}}
   V^{-\frac{\alpha+1}{q-1}}}}\\
   \nonumber&\quad\leq C\Big\{\abs{\alpha}^{-\frac{(p-1)q}{q-p+1}} \intst{\abs{\nabla
   \varphi}^{\frac{p\pa{q+\alpha}}{q-p+1}}V^{-\frac{p+\alpha-1}{q-p+1}}}\\
  \nonumber&\qquad +
\abs{\alpha}^{-\frac{(p-1)q}{q-p+1}}
\int\int_{E_{2^{1/\theta_2}nR}\setminus
E_{nR}}\varphi^{\frac{p\pa{q+\alpha}}{q-p+1}}V^{-\frac{p+\alpha-1}{q-p+1}}|\nabla\eta_n|^{\frac{p\pa{q+\alpha}}{q-p+1}}d\mu
dt\\
 \nonumber&\qquad +
\intst{\abs{\partial_t\varphi}^{\frac{q+\alpha}{q-1}}
V^{-\frac{\alpha+1}{q-1}}} + \int\int_{E_{2^{1/\theta_2}nR}\setminus
E_{nR}}{\varphi^{\frac{q+\alpha}{q-1}}
V^{-\frac{\alpha+1}{q-1}}}|\partial_t
\eta_n|^{\frac{q+\alpha}{q-1}}d\mu dt \Big\}\\
 \nonumber&\quad\leq  C \Big\{
 \abs{\alpha}^{-\frac{(p-1)q}{q-p+1}}(I_1 + I_2 ) + I_3 + I_4
 \Big\}\,,
\end{align}
where
\begin{eqnarray}\label{I1}
I_1& := & \intst{\abs{\nabla
   \varphi}^{\frac{p\pa{q+\alpha}}{q-p+1}}V^{-\frac{p+\alpha-1}{q-p+1}}}\,,\\ \label{I2}
I_2& := & \int\int_{E_{2^{1/\theta_2}nR}\setminus
E_{nR}}\varphi^{\frac{p\pa{q+\alpha}}{q-p+1}}V^{-\frac{p+\alpha-1}{q-p+1}}|\nabla\eta_n|^{\frac{p\pa{q+\alpha}}{q-p+1}}d\mu
dt\,,\\ \label{I3} I_3& := &
\intst{\abs{\partial_t\varphi}^{\frac{q+\alpha}{q-1}}
V^{-\frac{\alpha+1}{q-1}}}\,,\\ \label{I4} I_4& :=
&\int\int_{E_{2^{1/\theta_2}nR}\setminus
E_{nR}}{\varphi^{\frac{q+\alpha}{q-1}}
V^{-\frac{\alpha+1}{q-1}}}|\partial_t
\eta_n|^{\frac{q+\alpha}{q-1}}d\mu dt\,.
\end{eqnarray}

In view of \eqref{cutoff} and \eqref{cutoff2} and assumption {\bf
HP1}-$(ii)$ (see \eqref{hp1b}) with $\eps=-\frac{\alpha}{q-p+1}>0$,
for every $n\in \mathbb N$ and every small enough $|\alpha|>0$ we
get
\begin{eqnarray}
\label{eI2} I_2& \leq &  \int\int_{E_{2^{1/\theta_2}nR}\setminus
E_{nR}}\left[\theta_2 \left(\frac
1{nR}\right)^{\theta_2}r(x)^{\theta_2-1}|\nabla
r(x)|\right]^{\frac{p(q+\alpha)}{q-p+1}}n^{C_1\theta_2\alpha\frac{p(q+\alpha)}{q-p+1}}V^{-\frac{p+\alpha-1}{q-p+1}}d\mu
dt\\
\nonumber&\leq & C
(nR)^{-\frac{\theta_2p(q+\alpha)}{q-p+1}}n^{C_1\theta_2\alpha\frac{p(q+\alpha)}{q-p+1}}\int\int_{E_{2^{1/\theta_2}nR}\setminus
E_{nR}}r(x)^{(\theta_2-1)\frac{p(q+\alpha)}{q-p+1}}V^{-\frac{p+\alpha-1}{q-p+1}}d\mu
dt\\
\nonumber&\leq & C
(nR)^{-\frac{\theta_2p(q+\alpha)}{q-p+1}}n^{C_1\theta_2\alpha\frac{p(q+\alpha)}{q-p+1}}
(nR)^{\frac{\theta_2 p q}{q-p+1}+C_0\frac{|\alpha|}{q-p+1}}(\log
(nR))^{s_4}\,.
\end{eqnarray}
Now note that for any constant $\bar C\in\erre$ and for $R>0$ and
$\alpha=-\frac 1{\log R}$ we have
\begin{equation}\label{eq9}
R^{|\alpha| \bar C}=e^{|\alpha|\bar C\log R}=e^{\bar C}\leq C\,.
\end{equation}
Thus, also using the fact that
\[\frac{|\alpha|[\theta_2 p - C_1\theta_2 p(q+\alpha)+C_0]}{q-p+1}\leq - \frac{|\alpha|}{q-p+1}\,<\,0\,,\]
from \eqref{eI2} we deduce
\begin{equation}\label{e2I2}
I_2 \leq C n^{-\frac{|\alpha|}{q-p+1}}[\log(n R)]^{s_4}\,.
\end{equation}
In a similar way we can estimate $I_4$, using {\bf HP1}-$(i)$ (see
\eqref{hp1a}). Indeed, for $R>0$ large enough,
\begin{eqnarray}
\label{eI4} I_4&\leq & \int\int_{E_{2^{1/\theta_2}nR}\setminus
E_{nR}}\left(\frac{\theta_1}{(n R)^{\theta_2}}
t^{\theta_1-1}\right)^{\frac{q+\alpha}{q-1}}n^{C_1\theta_2\alpha\frac{q+\alpha}{q-1}}
V^{\frac{-1+|\alpha|}{q-1}} d\mu dt \\
\nonumber& \leq & C (n R)^{-\theta_2\frac{q+\alpha}{q-1}}
n^{C_1\theta_2\alpha\frac{q+\alpha}{q-1}}(n R)^{\frac{\theta_2
q+C_0|\alpha|}{q-1}}[\log(nR)]^{s_2}\\
\nonumber&\leq & C n^{\frac{\alpha}{q-1}\left(-\theta_2 +
C_1\theta_2 q- C_0 +C_1\theta_2
\alpha\right)}[\log(nR)]^{s_2}\,\,\leq\,\, C
n^{-\frac{|\alpha|}{q-1}}[\log(nR)]^{s_2}\,.
\end{eqnarray}

In order to estimate $I_1$ we observe that if $f:[0,\infty)\to
[0,\infty)$ is a nonincreasing function and if {\bf HP1}-$(ii)$
holds (see \eqref{hp1b}), then
\begin{equation}\label{eq6}
\int\int_{E_R^c}
f\big([r(x)^{\theta_2}+t^{\theta_1}]^{\frac1{\theta_2}}\big)r(x)^{(\theta_2-1)p\left(\frac
q{q-p+1}-\eps\right)}V^{-\frac{p-1}{q-p+1}+\eps}d\mu dt\leq
C\int_{R/2^{1/\theta_2}}^{\infty} f(z) z^{\bar s_3+ C_0 \eps
-1}(\log z)^{s_4} dz\,,
\end{equation}
for every $0<\eps<\eps_0$ and $R>0$ large enough. This can shown by
minor variations in the proof of \cite[formula (2.19)]{GrigS}.

Now, since for a.e. $x\in M$ we have $|\nabla r(x)|\leq 1$, we
obtain for a.e. $(x,t)\in S$
\begin{equation}\label{eq52}
|\nabla
\varphi(x,t)|\leq C_1
|\alpha|\theta_2\left(\frac{r(x)^{\theta_2}+t^{\theta_1}}{R^{\theta_2}}
\right)^{C_1\alpha-1}\frac{r(x)^{\theta_2-1}}{R^{\theta_2}}\,.
\end{equation}
Thus, using \eqref{eq9} for every sufficiently large $R>0$ we get
\begin{eqnarray*}
|\alpha|^{-\frac{(p-1)q}{q-p+1}} I_1 & \leq &  C
|\alpha|^{-\frac{(p-1)q}{q-p+1}}\int\int_{E_R^c}
V^{-\frac{p+\alpha-1}{q-p+1}}\left[ C_1
|\alpha|\theta_2\left(\frac{r(x)^{\theta_2}+t^{\theta_1}}{R^{\theta_2}}
\right)^{C_1\alpha-1}\frac{r(x)^{\theta_2-1}}{R^{\theta_2}}\right]^{\frac{p(q+\alpha)}{q-p+1}}d\mu
dt\\
&\leq & C
|\alpha|^{\frac{p(q+\alpha)-(p-1)q}{q-p+1}}\int\int_{E_R^c}\left\{\left[r(x)^{\theta_2}+t^{\theta_1}\right]^{\frac
1{\theta_2}}\right\}^{\theta_2\frac{(C_1\alpha-1)p(q+\alpha)}{q-p+1}}r(x)^{(\theta_2-1)p\frac{(q+\alpha)}{q-p+1}}V^{-\frac{p+\alpha-1}{q-p+1}}d\mu
dt\,.
\end{eqnarray*}
Now, using \eqref{eq6} with $\eps=\frac{|\alpha|}{q-p+1}$,
\begin{equation}\label{eq7}
|\alpha|^{-\frac{(p-1)q}{q-p+1}} I_1  \leq   C
|\alpha|^{\frac{p(q+\alpha)-(p-1)q}{q-p+1}}
\int_{R/2^{1/\theta_2}}^{\infty} z^{\theta_2(C_1\alpha
-1)\frac{p(q+\alpha)}{q-p+1}+\bar s_3 +
C_0\frac{|\alpha|}{q-p+1}-1}(\log z)^{s_4} dz\,.
\end{equation}
By our choice of $C_1$ and by the very definition of $\bar s_3$ we
have
\begin{equation*}
b:=\theta_2(C_1\alpha -1)\frac{p(q+\alpha)}{q-p+1}+\bar s_3 +
C_0\frac{|\alpha|}{q-p+1}\leq -\frac{|\alpha|}{q-p+1}\,.
\end{equation*} Then using the change of
variable $y:=|b|\log z$ in the right hand side of \eqref{eq7} we
obtain for $\alpha>0$ small enough
\begin{eqnarray}\label{eq8}
|\alpha|^{-\frac{(p-1)q}{q-p+1}} I_1 & \leq &  C
|\alpha|^{\frac{p(q+\alpha)-(p-1)q}{q-p+1}}\int_0^\infty
e^{-y}\left(\frac{y}{|b|}\right)^{s_4}\frac{dy}{|b|}\\
\nonumber &\leq & C
|\alpha|^{\frac{p(q+\alpha)-(p-1)q}{q-p+1}-s_4-1}\, \leq \, C
|\alpha|^{\frac{p-1}{q-p+1}-s_4}\,.
\end{eqnarray}

The term $I_3$ can be estimated similarly. Indeed, we start noting
that if $f:[0,\infty)\to [0,\infty)$ is nonincreasing function and
if {\bf HP1}-$(i)$ holds (see \eqref{hp1a}), then for any
sufficiently small $\eps>0$ and every large enough $R>0$ we get
\begin{equation}\label{eq8a}
\int\int_{E_R^c} f\big([r(x)^{\theta_2}+t^{\theta_1}]^{\frac
1{\theta_2}} \big)t^{(\theta_1-1)\left(\frac q{q-1}-\eps \right)}
V^{-\frac 1{q-1}+\eps} d\mu dt\leq C
\int_{R/2^{1/\theta_2}}^{\infty} f(z) z^{\bar s_1 + C_0\eps -1}(\log
z)^{s_2}dz\,;
\end{equation}
this can be shown by minor changes in the proof of \cite[formula
(2.19)]{GrigS}. Since for a.e. $(x,t)\in S$
\begin{equation}\label{eq30}
|\partial_t \varphi(x,t)|\leq C_1 |\alpha|\theta_1\left(\frac{r(x)^{\theta_2}+t^{\theta_1}}{R^{\theta_2}}
\right)^{C_1\alpha-1}\frac{t^{\theta_1-1}}{R^{\theta_2}},
\end{equation}
also using \eqref{eq9}, we have for every $R>0$ large enough
\begin{eqnarray*}
I_3& \leq & C
|\alpha|^{\frac{q+\alpha}{q-1}}\int\int_{E_R^c}\left[(r(x)^{\theta_2}+t^{\theta_1}
)^{\frac
1{\theta_2}}\right]^{\theta_2(C_1\alpha-1)\frac{q+\alpha}{q-1}}t^{(\theta_1-1)\frac{q+\alpha}{q-1}}V^{-\frac{1+\alpha}{q-1}}d\mu
dt\,.
\end{eqnarray*}
Now due to \eqref{eq8a} with $\eps=\frac{|\alpha|}{q-1}$
\begin{equation}\label{eq10}
I_3 \leq  C |\alpha|^{\frac{q+\alpha}{q-1}}
\int_{R/2^{1/\theta_2}}^\infty z^{\theta_2(C_1\alpha
-1)\frac{q+\alpha}{q-1}+ C_0\frac{|\alpha|}{q-1}+\bar s_1-1}(\log
z)^{s_2} dz\,.
\end{equation}
By our choice of $C_1$ and the very definition of $\bar s_1$ we
have that
\[\beta:= \theta_2(C_1\alpha -1 )\frac{q+\alpha}{q-1}+\bar s_1 + C_0\frac{|\alpha|}{q-1}\leq -\frac{|\alpha|}{q-1}\,.\]
Using the change of variable $y:=|\beta|\log z$ in \eqref{eq10} we
obtain
\begin{eqnarray}\label{eq11}
I_3 & \leq & C|\alpha|^{\frac q{q-1}}\int_0^\infty
e^{-y}\left(\frac{y}{|\beta|} \right)^{s_2} \frac{dy}{|\beta|}\\
\nonumber &\leq & C |\alpha|^{\frac 1{q-1}-s_2}\,.
\end{eqnarray}
Inserting \eqref{e2I2}, \eqref{eI4}, \eqref{eq8} and \eqref{eq11}
into \eqref{eq12} we obtain for every $n\in \mathbb N$ and every
sufficiently large $R>0$
\begin{eqnarray*}
\int\int_{E_R} u^{q+\alpha} V d\mu dt & \leq &
\intst{u^{q+\alpha}\varphi_n^s V} \\ &\leq & C
\left(|\alpha|^{\frac{p-1}{q-p+1}-s_4} +
|\alpha|^{-\frac{(p-1)q}{q-p+1}} n^{-\frac{|\alpha|}{p-q+1}}[\log(n
R)]^{s_4}+|\alpha|^{\frac
1{q-1}-s_2}+n^{-\frac{|\alpha|}{q-1}}[\log(nR)]^{s_2}\right)
\end{eqnarray*}
with $C$ independent of $n$ and $R$\,. Passing to the $\liminf$ as
$n\to\infty$ we deduce that
\begin{equation*}
\int\int_{E_R} u^{q+\alpha} V d\mu dt  \leq  C
\left(|\alpha|^{\frac{p-1}{q-p+1}-s_4} + |\alpha|^{\frac
1{q-1}-s_2}\right)\,.
\end{equation*}
Therefore, letting $R\to \infty$ (and thus $\alpha\to 0$), by
Fatou's lemma, we have
\begin{equation*}
\intst{ u^{q} V } =0
\end{equation*}
in view of our assumptions on $s_2,s_4$, which concludes the proof.
\end{proof}

\begin{proof}[\it Proof of Theorem \ref{thm2}]
  We claim that $u^q \in L^1\pa{S, V d\mu dt}$. To see this, we will show that
  \begin{equation}\label{eq17}
    \intst{u^q V} \leq A\pa{\intst{u^q V}}^\sigma + B
  \end{equation}
for some constants $A>0$, $B>0$, $0<\sigma<1$. In order to prove
\eqref{eq17} we consider \eqref{2.11} with $\varphi$ replaced by the
family of functions $\varphi_n$ defined in \eqref{eq13}, for any
fixed $s\geq \max\set{1, \frac{q+1}{q-1}, \frac{2 pq}{q-p+1}}$ and
$C_1
>\max\left\{\frac{1+C_0+\theta_2}{\theta_2},\frac{2(C_0+1)}{\theta_2(q-p+1)},\frac{2(C_0+1)}{\theta_2q(q-1)}\right\}$
with $C_0$, $\theta_2$ as in {\bf HP2} and with $R>0$ sufficiently
large and $\alpha=-\frac{1}{\log R}$. Thus we have
  \begin{align}
\label{eq14}&  \intst{\varphi_n^su^q V} \\
\nonumber& \leq C\left(
  |\alpha|^{-1-\frac{(p-1)q}{q-p+1}}\intst{{V^{-\frac{p+\alpha-1}{q-p+1}}\abs{\nabla\varphi_n}^{\frac{p(q+\alpha)}{q-p+1}}}}
  \right)^\frac{p-1}{p}\pa{\int\int_{E^c_R}\varphi_n^su^q
     V\,d\mu dt}^\frac{(1-\alpha)(p-1)}{qp}\\
\nonumber&\quad\times\pa{\int\int_{E^c_R}V^{-\frac{(1-\alpha)(p-1)}{q-(1-\alpha)(p-1)}}
     \abs{\nabla\varphi_n}^\frac{pq}{q-(1-\alpha)(p-1)}\,d\mu dt}^\frac{q-(1-\alpha)(p-1)}{qp}\\
\nonumber &\quad + C\left(|\alpha|^{-1}
  \intst{\abs{\partial_t\varphi_n}^{\frac{q+\alpha}{q-1}} V^{-\frac{\alpha+1}{q-1}}} \right)^\frac{p-1}{p}\pa{\int\int_{E^c_R}\varphi_n^su^q
     V\,d\mu dt}^\frac{(1-\alpha)(p-1)}{qp}\\
\nonumber&\quad\times\pa{\int\int_{E^c_R}V^{-\frac{(1-\alpha)(p-1)}{q-(1-\alpha)(p-1)}}
     \abs{\nabla\varphi_n}^\frac{pq}{q-(1-\alpha)(p-1)}\,d\mu
     dt}^\frac{q-(1-\alpha)(p-1)}{qp}\\
\nonumber &\quad +C\pa{\abs{\alpha}^{-\frac{(p-1)q}{q-p+1}}
\intst{\abs{\nabla\varphi_n}^{\frac{p\pa{q+\alpha}}{q-p+1}}
     V^{-\frac{p+\alpha-1}{q-p+1}}} + \intst{\abs{\partial_t\varphi_n}^{\frac{q+\alpha}{q-1}}
     V^{-\frac{\alpha+1}{q-1}}}}^{\frac{1}{q+\alpha}}\\
\nonumber &\quad\times\left(\intst{V^{-\frac
1{q+\alpha-1}}|\partial_t
  \varphi_n|^{\frac{q+\alpha}{q+\alpha-1}}}\right)^{\frac{q+\alpha-1}{q+\alpha}}\,.
\end{align}
Let us prove that for $R>0$ large enough, and thus for
$|\alpha|=\frac{1}{\log R}$ sufficiently small,
\begin{eqnarray}
\label{eq15a}
\limsup_{n\to\infty}\left(|\alpha|^{-\frac{(p-1)q}{q-p+1}}J_1\right)&\leq& C, \\
\label{eq15}
\limsup_{n\to\infty}\left(|\alpha|^{-\frac{(p-1)q}{q-(1-\alpha)(p-1)}}J_3\right)&\leq& C, \\
\label{eq16}
\limsup_{n\to\infty}J_2 &\leq& C, \\
%\label{eq18}
%|\alpha|^{-\frac{(p-1)q}{q-p+1}}J_4 &\leq& C, \\
\label{eq18} \limsup_{n\to\infty}J_4 &\leq& C,
\end{eqnarray}
for some $C>0$ independent of $\alpha$, where
\begin{eqnarray}\label{J1}
  J_1 &:=&\intst{{V^{-\frac{p+\alpha-1}{q-p+1}}\abs{\nabla\varphi_n}^{\frac{p(q+\alpha)}{q-p+1}}}}, \\
\label{J2}
  J_2 &:=& \intst{\abs{\partial_t\varphi_n}^{\frac{q+\alpha}{q-1}} V^{-\frac{\alpha+1}{q-1}}}, \\
\label{J3}
  J_3 &:=& \int\int_{E^c_R}V^{-\frac{(1-\alpha)(p-1)}{q-(1-\alpha)(p-1)}}
     \abs{\nabla\varphi_n}^\frac{pq}{q-(1-\alpha)(p-1)}\,d\mu dt, \\
%\label{J4}
%  J_4 &=& \intst{\abs{\nabla\varphi_n}^{\frac{p\pa{q+\alpha}}{q-p+1}}
%     V^{-\frac{p+\alpha-1}{q-p+1}}}, \\
\label{J4}
  J_4 &:=& \intst{V^{-\frac
1{q+\alpha-1}}|\partial_t
  \varphi_n|^{\frac{q+\alpha}{q+\alpha-1}}}.
\end{eqnarray}
Note that
\begin{equation}\label{eq19}
  J_1 \leq C(I_1+I_2),
\end{equation}
with $I_1$ and $I_2$ defined in \eqref{I1} and \eqref{I2},
respectively. Due to \eqref{hp2b} in {\bf HP2}-$(ii)$, by the same
arguments used to obtain \eqref{eq8} and \eqref{e2I2} with $s_4$
replaced by $\bar s_4$, for every $n\in \mathbb N, R>0$ large enough
and $\alpha=\frac 1{\log R}$ we have
\[|\alpha|^{-\frac{(p-1)q}{q-p+1}} J_1 \leq C\left(1+ |\alpha|^{-\frac{(p-1)q}{q-p+1}} n^{-\frac{|\alpha|}{q-p+1}}[\log(nR)]^{\bar s_4}\right)\,. \]
Letting $n\to \infty$ we get \eqref{eq15a}.

Next we observe that
\begin{equation}
  J_2 \leq C(I_3+I_4),
\end{equation}
with $I_3$ and $I_4$ defined in \eqref{I3} and \eqref{I4},
respectively. By the same computations used to obtain \eqref{eq11}
and \eqref{eI4}, with $s_2$ replaced by $\bar s_2$, we have for
every $n\in \mathbb N, R>0$ large enough and $\alpha=\frac 1{\log
R}$
\[J_2 \leq C\left(1+ n^{-\frac{|\alpha|}{q-1}}[\log(nR)]^{\bar s_2}\right)\,.\]
Again, letting $n\to \infty$ we obtain \eqref{eq16}.

We now proceed to estimate $J_4$; note that
\begin{equation}\label{eq21}
  J_4 \leq C(I_5+I_6),
\end{equation}
where
\begin{eqnarray}
  I_5 &:=& \intst{V^{-\frac
1{q+\alpha-1}}|\partial_t
  \varphi|^{\frac{q+\alpha}{q+\alpha-1}}}\\ I_6 &:=& \intst{V^{-\frac
1{q+\alpha-1}}\varphi^{\frac{q+\alpha}{q+\alpha-1}}|\partial_t
  \eta_n|^{\frac{q+\alpha}{q+\alpha-1}}}.
\end{eqnarray}
Due to \eqref{eq30} and \eqref{eq9}, we have for every $R>0$ large enough

\begin{align}\label{eq23}
I_5 \leq  C |\alpha|^{\frac{q+\alpha}{q+\alpha-1}}
\int\int_{E_R^c}\Big[(r(x)^{\theta_2}+t^{\theta_1})^{\frac
1{\theta_2}}&\Big]^{\theta_2(C_1\alpha-1)\left(\frac{q}{q-1}+\frac{|\alpha|}{(q+\alpha-1)(q-1)}\right)}
\\ \nonumber & \times
t^{(\theta_1-1)\left(\frac{q}{q-1}+\frac{|\alpha|}{(q+\alpha-1)(q-1)}\right)}
V^{-\frac{1}{q-1}-\frac{|\alpha|}{(q+\alpha-1)(q-1)}}d\mu dt\,.
\end{align}
Note that if $f:[0,\infty)\to [0,\infty)$ is a nonincreasing
function and if \eqref{hp2aa} in {\bf HP2}-$(i)$ holds, then for any
sufficiently small $\eps>0$ and every large enough $R>0$ we get
\begin{equation}\label{eq22}
\int\int_{E_R^c} f\big([r(x)^{\theta_2}+t^{\theta_1}]^{\frac
1{\theta_2}} \big)t^{(\theta_1-1)\left(\frac q{q-1}+\eps \right)}
V^{-\frac 1{q-1}-\eps} d\mu dt\leq C
\int_{R/2^{1/\theta_2}}^{\infty} f(z) z^{\bar s_1 + C_0\eps -1}(\log
z)^{\bar s_2}dz\,;
\end{equation}
this can be shown by minor changes in the proof of \cite[formula
(2.19)]{GrigS}. Now due to \eqref{eq23}, \eqref{eq22} with
$\eps=\frac{|\alpha|}{(q+\alpha-1)(q-1)}$
\begin{equation}\label{eq24}
I_5 \leq  C |\alpha|^{\frac{q+\alpha}{q+\alpha-1}}
\int_{R/2^{1/\theta_2}}^\infty z^{\theta_2(C_1\alpha
-1)\frac{q+\alpha}{q+\alpha-1}+
C_0\frac{|\alpha|}{(q+\alpha-1)(q-1)}+\bar s_1-1}(\log z)^{\bar s_2}
dz\,.
\end{equation}
By our choice of $C_1$ and the very definition of $\bar s_1$, we
have for sufficiently small $|\alpha|>0$
\[\gamma:=\theta_2(C_1\alpha -1)\frac{q+\alpha}{q+\alpha-1}+
C_0\frac{|\alpha|}{(q+\alpha-1)(q-1)}+\bar
s_1<-\frac{|\alpha|}{(q-1)^2}\,.
\]
Using the change of variable $y:=|\gamma|\log z$ in the right hand
side of \eqref{eq24}, due to the very definition of $\bar s_2$ we
obtain for every $R>0$ large enough
\begin{equation}\label{eq25}
I_5  \leq  C|\alpha|^{\frac {q+\alpha}{q+\alpha-1}}\int_0^\infty
e^{-y}\left(\frac{y}{|\gamma|} \right)^{\bar s_2}
\frac{dy}{|\gamma|}\leq  C \,.
\end{equation}
Moreover, using \eqref{eq9} and \eqref{hp2aa} in {\bf HP2}, for
every $n\in\enne$ and $\alpha>0$ sufficiently small, we have
\begin{eqnarray}
\label{eq26} I_6&\leq & \int\int_{E_{2^{1/\theta_2}nR}\setminus
E_{nR}}\left(\frac{\theta_1}{(n R)^{\theta_2}}
t^{\theta_1-1}\right)^{\frac{q+\alpha}{q+\alpha-1}}n^{C_1\theta_2\alpha\frac{q+\alpha}{q+\alpha-1}}
V^{-\frac{1}{q-1}-\frac{|\alpha|}{(q+\alpha-1)(q-1)}} d\mu dt \\
\nonumber& \leq & C (n R)^{-\theta_2\frac{q+\alpha}{q+\alpha-1}}
n^{C_1\theta_2\alpha\frac{q+\alpha}{q+\alpha-1}}(n
R)^{\frac{\theta_2
q}{q-1}+ C_0\frac{|\alpha|}{(q+\alpha-1)(q-1)}}[\log(nR)]^{\bar s_2}\\
\nonumber&\leq & C n^{-\frac{|\alpha|}{(q-1)^2}}[\log(nR)]^{\bar
s_2}\,.
\end{eqnarray}
In view of \eqref{eq21}, \eqref{eq25}, \eqref{eq26} we obtain
\[J_4\leq C \left( 1 + n^{-\frac{|\alpha|}{(q-1)^2}}[\log(nR)]^{\bar
s_2} \right)\,.\] Letting $n\to \infty$ we get \eqref{eq18}.

%\left[\frac{r(x)^{\theta_2}+t^{\theta_1}}{R^{\theta_2}}\right]^{{C_1\alpha-1}\frac{t^{\theta_1-1}}{R^{\theta_2}}{\frac{q+\alpha}{q+\alpha-1}}}

In order to estimate the integral $J_3$ we start by defining
$\Lambda = \frac{(p-1)q
|\alpha|}{(q-p+1)\sq{q-\pa{1-\alpha}\pa{p-1}}}$, and we note that
\begin{equation}\label{eq40}
   \frac{(p-1)q}{\pa{q-p+1}^2}|\alpha|<\Lambda<\frac{2(p-1)q}{\pa{q-p+1}^2}|\alpha|<\eps^*
\end{equation}
for every small enough $|\alpha|>0$, and that
  \[\frac{(1-\alpha)(p-1)}{q-\pa{1-\alpha}\pa{p-1}}=\bar s_4+\Lambda
  \qquad\text{ and }\qquad\frac{p q}{q-\pa{1-\alpha}\pa{p-1}}=\frac{\bar s_3}{\theta_2} +\Lambda
  p\,.
  \]
By our definition of the functions $\varphi_n$, for every
$n\in\enne$ and every small enough $|\alpha|>0$ we have
\begin{align}
\label{eq41}  J_3& =\intst{V^{-\bar s_4-\Lambda}\abs{\nabla
\varphi_n}^{\frac{\bar s_3}{\theta_2}+\Lambda p}}\\ \nonumber & \leq
    C\sq{\intst{V^{-\bar s_4-\Lambda}{\eta_n}^{\frac{\bar s_3}{\theta_2}+\Lambda p}\abs{\nabla \varphi}^{\frac{\bar s_3}{\theta_2}+\Lambda
    p}}
      +\intst{V^{-\bar s_4-\Lambda} \varphi^{\frac{\bar s_3}{\theta_2}+\Lambda p}\abs{\nabla \eta_n}^{\frac{\bar s_3}{\theta_2}+\Lambda p}}}\\
  \nonumber  &\leq C\sq{\int\int_{E_R^c} V^{-\bar s_4-\Lambda} \abs{\nabla \varphi}^{\frac{\bar s_3}{\theta_2}+\Lambda
  p}\, d\mu dt
      +\int\int_{E_{2^{1/\theta_2}nR}\setminus E_{nR}} V^{-\bar s_4-\Lambda} \varphi^{\frac{\bar s_3}{\theta_2}+\Lambda p}\abs{\nabla\eta_n}^{\frac{\bar s_3}{\theta_2}+\Lambda p}\,d\mu dt}\\
   \nonumber &:=C\pa{I_7+I_8}.
\end{align}
Now we use condition \eqref{hp2bb} in {\bf HP2}-$(ii)$ with
$\eps=\Lambda$, and we obtain for every $n\in\enne$ and $R>0$ large
enough
   \begin{align*}
    I_8 & = \int\int_{E_{2^{1/\theta_2}nR}\setminus E_{nR}} V^{-\bar s_4-\Lambda} \varphi^{\frac{\bar s_3}{\theta_2}+\Lambda p}\abs{\nabla \eta_n}^{\frac{\bar s_3}{\theta_2}+\Lambda
    p}\,d\mu dt \\
       &\leq \pa{\sup_{E_{2^{1/\theta_2}nR}\setminus E_{nR}}\varphi}^{\frac{\bar s_3}{\theta_2}+\Lambda p}\int\int_{E_{2^{1/\theta_2}nR}\setminus E_{nR}}\pa{\frac{\theta_2 r(x)^{\theta_2-1}}
       {(nR)^{\theta_2}}}^{\frac{\bar s_3}{\theta_2}+\Lambda p}
           V^{-\bar s_4-\Lambda}\,d\mu dt \\
    &\leq C n^{\frac{C_1\theta_2\alpha pq}{q-(1-\alpha)(p-1)}}(nR)^{-\frac{\theta_2 pq}{q-(1-\alpha)(p-1)}}\int\int_{E_{2^{1/\theta_2}nR}\setminus E_{nR}}
    r(x)^{(\theta_2-1)p\left(\frac{q}{q-p+1}+\Lambda\right)}V^{-\frac{p-1}{q-p+1}-\Lambda}d\mu
dt  \\
    &\leq C n^{\frac{C_1\theta_2\alpha pq}{q-(1-\alpha)(p-1)}}(nR)^{-\frac{\theta_2 pq}{q-(1-\alpha)(p-1)}} (nR)^{\frac{\theta_2 pq}{q-p+1} + C_0\Lambda} (\log(nR))^{\bar s_4} .
   \end{align*}
By our definition of $C_1,\Lambda$ and by relation \eqref{eq40} we
easily find
   \begin{align}
     \label{eq42}  & \frac{C_1\theta_2\alpha pq}{q-(1-\alpha)(p-1)} -\frac{\theta_2 pq}{q-(1-\alpha)(p-1)}+ \frac{\theta_2 pq}{q-p+1} + C_0\Lambda
      \\ \nonumber & \qquad\qquad < \frac{pq \alpha C_1\theta_2}{q-\pa{1-\alpha}\pa{p-1}}-\frac{ \alpha q(p-1)C_0}{\sq{q-\pa{1-\alpha}\pa{p-1}}\pa{q-p+1}} \\
     \nonumber& \qquad\qquad \leq \frac{q \alpha p}{\sq{q-\pa{1-\alpha}\pa{p-1}}\pa{q-p+1}} < \frac{q \alpha p}{\pa{q-p+1}^2} <0,
   \end{align}
for any small enough $|\alpha|>0$. Moreover by \eqref{eq9}, since
$\alpha=-\frac{1}{\log R}$, we have
\[
R^{-\frac{\theta_2 pq}{q-(1-\alpha)(p-1)}+\frac{\theta_2
pq}{q-p+1} + C_0\Lambda }\leq C\,.
\]
Thus, for any sufficiently large $R>0$ and every $n\in\enne$,
\begin{equation}\label{eq44}
    I_8 \leq C n^{ \frac{q \alpha p}{\pa{q-p+1}^2}}\pa{\log\pa{nR}}^{\bar s_4}.
\end{equation}
In order to estimate $I_7$ we observe that if $f:[0,\infty)\to
[0,\infty)$ is a nonincreasing function and if {\bf HP2}-$(ii)$
holds (see \eqref{hp2bb}), then
\begin{equation}\label{eq43}
\int\int_{E_R^c}
f\big([r(x)^{\theta_2}+t^{\theta_1}]^{\frac1{\theta_2}}\big)r(x)^{(\theta_2-1)p\left(\frac
q{q-p+1}+\eps\right)}V^{-\frac{p-1}{q-p+1}-\eps}\,d\mu dt\leq
C\int_{R/2^{1/\theta_2}}^{\infty} f(z) z^{\bar s_3+ C_0 \eps
-1}(\log z)^{\bar s_4}\, dz\,,
\end{equation}
for every $0<\eps<\eps_0$ and $R>0$ large enough. This can againbe
shown by minor variations in the proof of \cite[formula
(2.19)]{GrigS}. Thus, similarly to \eqref{eq23} and \eqref{eq24},
using \eqref{eq9}, \eqref{eq40} and \eqref{eq43}, we have for $R>0$
large enough and $\alpha=-\frac{1}{\log R}$
\begin{align}
\label{eq45}I_7 & \leq C |\alpha|^{\frac{p
q}{q-(1-\alpha)(p-1)}}\int\int_{E_R^c}
\big[(r(x)^{\theta_2}+t^{\theta_1})^{\frac 1{\theta_2}}
\big]^{\theta_2 (C_1\alpha -1)\frac{p
q}{q-(1-\alpha)(p-1)}}r(x)^{(\theta_2-1)p\left(\frac
q{q-p+1}+\Lambda \right)} V^{-\frac{p-1}{q-p+1}-\Lambda} \, d\mu dt
\\ \nonumber & \leq C |\alpha|^{\frac{p
q}{q-(1-\alpha)(p-1)}}\int_{R/2^{1/\theta_2}}^\infty z^{\theta_2
(C_1\alpha -1)\frac{p
q}{q-(1-\alpha)(p-1)}+\bar{s}_3+C_0\Lambda-1}(\log
z)^{\bar{s}_4}\,dz.
\end{align}
By our choice of $C_1$ and the definition of $\bar{s}_3$ and
$\Lambda$ we have
\[
a:=\theta_2 (C_1\alpha -1)\frac{p
q}{q-(1-\alpha)(p-1)}+\bar{s}_3+C_0\Lambda<-\frac{qp|\alpha|}{(q-p+1)^2}<0,
\]
thus using the change of variable $y=|a|\log z$ in the last integral
in \eqref{eq45} we obtain
\begin{align}
\label{eq46}I_7 &\leq\,\, C |\alpha|^{\frac{p
q}{q-(1-\alpha)(p-1)}}\int_0^\infty
e^{-y}\left(\frac{y}{|a|}\right)^{\bar{s}_4}\,\frac{dy}{|a|}\\
\nonumber&\leq\,\, C |\alpha|^{\frac{p
q}{q-(1-\alpha)(p-1)}-\bar{s}_4-1}\,\,=\,\,C
|\alpha|^{\frac{(p-1)q}{q-(1-\alpha)(p-1)}+\frac{(p-1)q|\alpha
|}{(q-p+1)(q-(1-\alpha)(p-1))}}\,.
\end{align}
Thus for any sufficiently large $R>0$ and every $n\in\enne$, by
\eqref{eq41}, \eqref{eq44} and \eqref{eq46}
\[
|\alpha|^{-\frac{(p-1)q}{q-(1-\alpha)(p-1)}}J_3\leq
C|\alpha|^{-\frac{(p-1)q}{q-(1-\alpha)(p-1)}}\left(|\alpha|^{\frac{(p-1)q}{q-(1-\alpha)(p-1)}+\frac{(p-1)q|\alpha
|}{(q-p+1)(q-(1-\alpha)(p-1))}}+n^{ \frac{q \alpha
p}{\pa{q-p+1}^2}}\pa{\log\pa{nR}}^{\bar s_4}\right)\,.
\]
Letting $n\to \infty$, for every $R>0$ large enough and
$\alpha=-\frac{1}{\log R}$ we obtain
\[
|\alpha|^{-\frac{(p-1)q}{q-(1-\alpha)(p-1)}}J_3\leq
C|\alpha|^{\frac{(p-1)q|\alpha |}{(q-p+1)(q-(1-\alpha)(p-1))}}\leq
C,
\]
that is \eqref{eq15}.

Now using \eqref{eq15a}--\eqref{eq18} in \eqref{eq14}, since
$\varphi_n\equiv1$ on $E_R$ and $0\leq\varphi_n\leq1$ on
$M\times[0,\infty)$, for every $R>0$ large enough we have
\[
\int\int_{E_R}u^qV\,d\mu
dt\,\,\leq\,\,\limsup_{n\to\infty}\left(\intst{\varphi_n^su^qV}\right)\,\,\leq
A\left(\intst{u^qV}\right)^\sigma+B
\]
for some positive constants $A,B$ and $\sigma\in(0,1)$. Passing to
the limit as $R\to\infty$ we obtain \eqref{eq17}, and hence we
conclude that $u^q \in L^1\pa{S, V d\mu dt}$ as claimed.

Next we want to show that
\[
\intst{u^qV}=0,
\]
and thus that $u=0$ a.e., since $V>0$ a.e. on $M\times[0,\infty)$.
To this aim, we consider \eqref{eq47} with $\varphi$ replaced by the
family of functions $\varphi_n$. Since $\varphi_n\equiv1$ on $E_R$
and since $0\leq\varphi_n\leq1$ on $M\times[0,\infty)$, for every
$n\in\enne$, every $R>0$ large enough and $\alpha=-\frac{1}{\log R}$
we have
\begin{align}
  \label{eq48}&\int\int_{E_R}u^q V\,d\mu dt\,\,\leq\,\,\intst{\varphi_n^su^q V}\\
  \nonumber&\quad\leq C
     \pa{\abs{\alpha}^{-1-\frac{(p-1)q}{(q-p+1)}}\intst{{V^{-\frac{p+\alpha-1}{q-p+1}}\abs{\nabla\varphi_n}^{\frac{p(q+\alpha)}{q-p+1}}}}+|\alpha|^{-1} \intst{\abs{\partial_t\varphi_n}^{\frac{q+\alpha}{q-1}} V^{-\frac{\alpha+1}{q-1}}} }^\frac{p-1}{p}\\
  \nonumber&\qquad\times\pa{\int\int_{E^c_R} V^{-\frac{(1-\alpha)(p-1)}{q-(1-\alpha)(p-1)}}
     \abs{\nabla\varphi_n}^\frac{pq}{q-(1-\alpha)(p-1)}\,d\mu dt}^\frac{q-(1-\alpha)(p-1)}{pq}\pa{\int\int_{E^c_R}u^q V\,d\mu dt}^\frac{(1-\alpha)(p-1)}{pq} \\
  \nonumber&\qquad+C\left(\int\int_{E^c_R}u^{q}V\,d\mu dt\right)^\frac{1}{q}\pa{\intst{V^{-\frac{1}{q-1}}\abs{\partial_t\varphi_n}^{\frac{q}{q-1}}}}^{\frac{q-1}{q}}.
\end{align}
Now we claim that for $R>0$ sufficiently large
\begin{equation}\label{eq50}
\limsup_{n\to\infty}J_5\leq C,
\end{equation}
where
\[
J_5:=\intst{V^{-\frac{1}{q-1}}\abs{\partial_t\varphi_n}^{\frac{q}{q-1}}}.
\]
This can be shown similarly to inequality \eqref{eq18}. Indeed
\begin{equation}\label{eq71}
  J_5 \leq C(I_9+I_{10}),
\end{equation}
where
\begin{equation*}
  I_9 \,:=\, \intst{V^{-\frac
1{q-1}}|\partial_t
  \varphi|^{\frac{q}{q-1}}}\,,\qquad I_{10}\, :=\, \intst{V^{-\frac
1{q-1}}\varphi^{\frac{q}{q-1}}|\partial_t
  \eta_n|^{\frac{q}{q-1}}}.
\end{equation*}
By \eqref{eq30} and \eqref{eq9}, for $R>0$ sufficiently large
\begin{align}\label{eq73}
I_9 \leq  C |\alpha|^{\frac{q}{q-1}}
\int\int_{E_R^c}\Big[(r(x)^{\theta_2}+t^{\theta_1})^{\frac
1{\theta_2}}&\Big]^{\theta_2(C_1\alpha-1)\frac{q}{q-1}}
t^{(\theta_1-1)\frac{q}{q-1}} V^{-\frac{1}{q-1}}d\mu dt\,.
\end{align}
Now note that if $f:[0,\infty)\to [0,\infty)$ is a nonincreasing
function and if \eqref{hp2aa} in {\bf HP2}-$(i)$ holds, then for
every $R>0$ sufficiently large we get
\begin{equation}\label{eq72}
\int\int_{E_R^c} f\big([r(x)^{\theta_2}+t^{\theta_1}]^{\frac
1{\theta_2}} \big)t^{(\theta_1-1)\frac q{q-1}} V^{-\frac 1{q-1}}
d\mu dt\leq C \int_{R/2^{1/\theta_2}}^{\infty} f(z) z^{\bar
s_1-1}(\log z)^{\bar s_2}dz\,;
\end{equation}
indeed, the proof of \eqref{eq72} is similar to that of
\cite[formula (2.19)]{GrigS}, where here one uses condition
\eqref{hp2aa} with $\eps=0$, see also Remark \ref{rem1}. Then
\begin{align}\label{eq74}
I_9 &\leq  C |\alpha|^{\frac{q}{q-1}} \int_{R/2^{1/\theta_2}}^\infty
z^{\theta_2(C_1\alpha -1)\frac{q}{q-1} +\bar s_1-1}(\log z)^{\bar
s_2} dz\\
\nonumber&\leq C |\alpha|^{\frac{q}{q-1}} \int_1^\infty
z^{\frac{\theta_2C_1\alpha q}{q-1}}(\log z)^{\bar
s_2}\,\frac{dz}{z}\,\leq\, C |\alpha|^{\frac{q}{q-1}-\bar
s_2-1}\,\leq\, C\,.
\end{align}
Moreover, for every $n\in\enne$ by \eqref{eq9} and by \eqref{hp2a}
with $\eps=0$, see also Remark \ref{rem1},
\begin{eqnarray}
\label{eq76} I_{10}&\leq & \int\int_{E_{2^{1/\theta_2}nR}\setminus
E_{nR}}\left(\frac{\theta_1}{(n R)^{\theta_2}}
t^{\theta_1-1}\right)^{\frac{q}{q-1}}n^{C_1\theta_2\alpha\frac{q}{q-1}}
V^{-\frac{1}{q-1}} d\mu dt \\
\nonumber& \leq & C (n R)^{-\theta_2\frac{q}{q-1}}
n^{C_1\theta_2\alpha\frac{q}{q-1}}(n R)^{\frac{\theta_2
q}{q-1}}[\log(nR)]^{\bar s_2}\,=\, C
n^{-\frac{C_1\theta_2q}{q-1}|\alpha|}[\log(nR)]^{\bar s_2}\,.
\end{eqnarray}
In view of \eqref{eq71}, \eqref{eq74}, \eqref{eq76} we have
\[J_5\leq C \left( 1 + n^{-\frac{C_1\theta_2q}{q-1}|\alpha|}[\log(nR)]^{\bar
s_2} \right)\,.\] Letting $n\to \infty$ we get our claim, inequality
\eqref{eq50}.

Now consider again \eqref{eq48}; passing to the limsup as
$n\to\infty$ and using \eqref{eq15a}--\eqref{eq16} and \eqref{eq50},
we obtain for some constant $C>0$
\begin{equation}
  \label{eq49}\int\int_{E_R}u^q V\,d\mu dt \leq
  C\sq{\pa{\int\int_{E^c_R}u^q V\,d\mu dt}^\frac{(1-\alpha)(p-1)}{pq}+\left(\int\int_{E^c_R}u^{q}V\,d\mu
  dt\right)^\frac{1}{q}}\,.
\end{equation}
Now we can pass to the limit in \eqref{eq49} as $R\to\infty$, and
thus as $\alpha\to 0$, and conclude by using Fatou's Lemma and the
fact that $u^q\in L^1(S,Vd\mu dt)$ that
\[
\intst{u^qV}=0.
\]
Thus $u=0$ a.e. on $M\times[0,\infty)$.

\end{proof}

\bigskip

\section{Proof of Corollaries \ref{cor1}, \ref{cor2} and
\ref{cor3}}\label{sec4}

\begin{proof}[\it Proof of Corollary \ref{cor1}\,.]
We now show that under our assumptions hypothesis {\bf HP1} is
satisfied (see conditions \eqref{hp1a} and \eqref{hp1b}). Observe
that for small $\eps>0$
\begin{equation*}
\int\int_{E_{2^{1/\theta_2}R}\setminus E_R}
t^{(\theta_1-1)\left(\frac q{q-1}-\eps \right)}\,dx dt \,\leq\, C
R^m
\int_0^{2^{1/\theta_1}R^{\frac{\theta_2}{\theta_1}}}t^{(\theta_1-1)\left(\frac
q{q-1}-\eps \right)} dt \, \leq\,  C R^m
R^{\frac{\theta_2}{\theta_1}\left[(\theta_1-1)\left(\frac
q{q-1}-\eps \right)+1\right]}\,.
\end{equation*}
Hence, condition \eqref{hp1a} is satisfied, if
\begin{equation}\label{eq53}
\frac{\theta_2}{\theta_1}\geq (q-1) m \,.
\end{equation}
On the other hand, for small $\eps>0$,
\begin{align*}
\int\int_{E_{2^{1/\theta_2}R}\setminus E_R}
|x|^{(\theta_2-1)p\left(\frac q{q-p+1}-\eps \right)}\,dx dt &\leq C
R^{\frac{\theta_2}{\theta_1}}
\int_0^{2^{1/\theta_2}R}\varrho^{(\theta_2-1)p\left(\frac
q{q-p+1}-\eps \right)+m-1} d\varrho \\
\nonumber & \leq  C R^{\frac{\theta_2}{\theta_1}+
\left[(\theta_2-1)p\left(\frac q{q-p+1}-\eps \right)+m\right]}\,.
\end{align*}
Therefore condition \eqref{hp1b} is satisfied, if
\begin{equation}\label{eq54}
\frac{\theta_2}{\theta_1} \leq \frac{p q}{q-p+1}-m\,.
\end{equation}
Now note that we can find $\theta_1\geq 1, \theta_2\geq 1$ such that
conditions \eqref{eq53} and \eqref{eq54} hold simultaneously, if
\eqref{eq55} holds. Thus, from Theorem \ref{thm1} the conclusion
follows.
\end{proof}

\begin{proof}[\it Proof of Corollary \ref{cor2}\,.]
Under our assumptions, for $R>0$ large and $\eps>0$ small enough we
have
\[
\int\int_{E_{2^{1/\theta_2R}}\setminus
E_{R}}t^{(\theta_1-1)\left(\frac{q}{q-1}-\eps\right)}V^{-\frac{1}{q-1}+\eps}\,d\mu
dt\leq
CR^{\frac{\theta_2}{\theta_1}(\theta_1-1)\left(\frac{q}{q-1}-\eps\right)+\frac{\theta_2}{\theta_1}\alpha\eps+\beta\eps+\frac{\theta_2}{\theta_1}\sigma_2+\sigma_1}(\log
R)^{\delta_1+\delta_2}.
\]
Hence condition \eqref{hp1a} in {\bf HP1} is satisfied if we choose
$C_0\geq\max\left\{0,\frac{\theta_2}{\theta_1}\left(\alpha+1\right)+\beta-\theta_2\right\}$
and if
\begin{equation}\label{eq57}
\frac{\theta_2}{\theta_1}\left(\sigma_2-\frac{q}{q-1}\right)+\sigma_1\leq0\,,\qquad
\delta_1+\delta_2<\frac{1}{q-1}\,.
\end{equation}
Similarly for sufficiently large $R>0$ and small $\eps>0$ we have
\begin{equation*}
\int\int_{E_{2^{1/\theta_2}R}\setminus E_R}
r(x)^{(\theta_2-1)p\left(\frac{q}{q-p+1}-\eps\right)}V^{-\frac{p-1}{q-p+1}+\eps}d\mu
dt\leq
CR^{(\theta_2-1)p\left(\frac{q}{q-p+1}-\eps\right)+\frac{\theta_2}{\theta_1}\alpha\eps+\beta\eps+\frac{\theta_2}{\theta_1}\sigma_4+\sigma_3}(\log
R)^{\delta_3+\delta_4}\,.
\end{equation*}
Therefore condition \eqref{hp1b} in {\bf HP1} is satisfied  if
$C_0\geq\max\left\{0,\beta+\frac{\theta_2}{\theta_1}\alpha-(\theta_2-1)p\right\}$
and if
\begin{equation}\label{eq58}
\left(-\frac{pq}{q-p+1}+\sigma_3\right)+\frac{\theta_2}{\theta_1}\sigma_4\leq0\,,\qquad
\delta_3+\delta_4<\frac{p-1}{q-p+1}\,.
\end{equation}
Now for conditions \eqref{eq57} and \eqref{eq58} to be satisfied, by
our assumptions it is sufficient to choose $\theta_1\geq1$,
$\theta_2\geq1$ such that
\begin{eqnarray}
\label{51}&&\sigma_1\left(\frac{q}{q-1}-\sigma_2\right)^{-1}\,\leq\,\frac{\theta_2}{\theta_1}\qquad\textrm{
if }0\leq\sigma_2<\frac{q}{q-1}\,,\\
\label{52}&&\frac{\theta_2}{\theta_1}\,\leq\,\left(\frac{pq}{q-p+1}-\sigma_3\right)\sigma_4^{-1}\qquad\textrm{
if }0\leq\sigma_3<\frac{pq}{q-p+1}\,.
\end{eqnarray}
Thus we can apply Theorem \ref{thm1} and conclude.
\end{proof}

\begin{proof}[\it Proof of Corollary \ref{cor3}\,.]
By our assumptions for large $R>0$ and small $\eps>0$ we have
\begin{eqnarray*}
\int\int_{E_{2^{1/\theta_2}R}\setminus
E_{R}}t^{(\theta_1-1)\left(\frac{q}{q-1}-\eps\right)}V^{-\frac{1}{q-1}+\eps}\,d\mu
dt&\leq&
CR^{\frac{\theta_2}{\theta_1}(\theta_1-1)\left(\frac{q}{q-1}-\eps\right)+\frac{\theta_2}{\theta_1}\alpha\eps+\beta\eps+\frac{\theta_2}{\theta_1}\sigma_2+\sigma_1}(\log
R)^{\delta_1+\delta_2}\,,\\
\int\int_{E_{2^{1/\theta_2}R}\setminus
E_{R}}t^{(\theta_1-1)\left(\frac{q}{q-1}+\eps\right)}V^{-\frac{1}{q-1}-\eps}\,d\mu
dt&\leq&
CR^{\frac{\theta_2}{\theta_1}(\theta_1-1)\left(\frac{q}{q-1}+\eps\right)+\frac{\theta_2}{\theta_1}\alpha\eps+\beta\eps+\frac{\theta_2}{\theta_1}\sigma_2+\sigma_1}(\log
R)^{\delta_1+\delta_2}\, .
\end{eqnarray*}
Thus conditions \eqref{hp2a}--\eqref{hp2aa} of {\bf HP2} are
satisfied if we choose
$C_0\geq\max\left\{0,\frac{\theta_2}{\theta_1}\left(\alpha-1\right)+\beta+\theta_2\right\}$
and
\begin{equation}\label{eq61}
\frac{\theta_2}{\theta_1}\left(\sigma_2-\frac{q}{q-1}\right)+\sigma_1\leq0\,,\qquad
\delta_1+\delta_2\leq\frac{1}{q-1}\,.
\end{equation}
Similarly if $R>0$ is large and $\eps>0$ is small enough we have
\begin{eqnarray*}
\int\int_{E_{2^{1/\theta_2}R}\setminus E_R}
r(x)^{(\theta_2-1)p\left(\frac{q}{q-p+1}-\eps\right)}V^{-\frac{p-1}{q-p+1}+\eps}d\mu
dt&\leq&
CR^{(\theta_2-1)p\left(\frac{q}{q-p+1}-\eps\right)+\frac{\theta_2}{\theta_1}\alpha\eps+\beta\eps+\frac{\theta_2}{\theta_1}\sigma_4+\sigma_3}(\log
R)^{\delta_3+\delta_4}\,,\\
\int\int_{E_{2^{1/\theta_2}R}\setminus E_R}
r(x)^{(\theta_2-1)p\left(\frac{q}{q-p+1}+\eps\right)}V^{-\frac{p-1}{q-p+1}-\eps}d\mu
dt&\leq&
CR^{(\theta_2-1)p\left(\frac{q}{q-p+1}+\eps\right)+\frac{\theta_2}{\theta_1}\alpha\eps+\beta\eps+\frac{\theta_2}{\theta_1}\sigma_4+\sigma_3}(\log
R)^{\delta_3+\delta_4}\,.
\end{eqnarray*}
Thus conditions \eqref{hp2b}--\eqref{hp2bb} in {\bf HP2} are
satisfied if
$C_0\geq\max\left\{0,\beta+\frac{\theta_2}{\theta_1}\alpha+(\theta_2-1)p\right\}$
and
\begin{equation}\label{eq62}
\left(-\frac{pq}{q-p+1}+\sigma_3\right)+\frac{\theta_2}{\theta_1}\sigma_4\leq0\,,\qquad
\delta_3+\delta_4\leq\frac{p-1}{q-p+1}\,.
\end{equation}
Hence, arguing as in the proof of Corollary \ref{cor2}, we have that
under our assumptions {\bf HP2} holds, and we can apply Theorem
\ref{thm2} to conclude.
\end{proof}

We conclude with the next example, where we show that our results
extend those in \cite{Zhang} in the case of the Laplace--Beltrami
operator on a complete noncompact manifold $M$.

Let us start by fixing a point $o\in M$ and denote by
$\textrm{Cut}(o)$ the {\it cut locus} of $o$. For any $x\in
M\setminus \big[\textrm{Cut}(o)\cup \{o\} \big]$, one can define the
{\it polar coordinates} with respect to $o$, see e.g. \cite{Grig}.
Namely, for any point $x\in M\setminus \big[\textrm{Cut}(o)\cup
\{o\} \big]$ there correspond a polar radius $r(x) := dist(x, o)$
and a polar angle $\theta\in \mathbb S^{m-1}$ such that the shortest
geodesics from $o$ to $x$ starts at $o$ with the direction $\theta$
in the tangent space $T_oM$. Since we can identify $T_o M$ with
$\mathbb R^m$, $\theta$ can be regarded as a point of $\mathbb
S^{m-1}.$

The Riemannian metric in $M\setminus\big[\textrm{Cut}(o)\cup \{o\}
\big]$ in polar coordinates reads
\[ds^2 = dr^2+A_{ij}(r, \theta)d\theta^i d\theta^j, \]
where $(\theta^1, \ldots, \theta^{m-1})$ are coordinates in
$\mathbb S^{m-1}$ and $(A_{ij})$ is a positive definite matrix. It
is not difficult to see that the Laplace-Beltrami operator in
polar coordinates has the form
\[
\Delta = \frac{\partial^2}{\partial r^2} + \mathcal F(r,
\theta)\frac{\partial}{\partial r}+\Delta_{S_{r}},
\]
where $\mathcal F(r, \theta):=\frac{\partial}{\partial
r}\big(\log\sqrt{A(r,\theta)}\big)$, $A(r,\theta):=\det
(A_{ij}(r,\theta))$, $\Delta_{S_r}$ is the Laplace-Beltrami
operator on the submanifold $S_{r}:=\partial B(o, r)\setminus
\textrm{Cut}(o)$\,.

$M$ is a {\it manifold with a pole}, if it has a point $o\in M$
with $\textrm{Cut}(o)=\emptyset$. The point $o$ is called {\it
pole} and the polar coordinates $(r,\theta)$ are defined in
$M\setminus\{o\}$.

A manifold with a pole is a {\it spherically symmetric manifold}
or a {\it model}, if the Riemannian metric is given by
\begin{equation}\label{eq600}
ds^2 = dr^2+\psi^2(r)d\theta^2,
\end{equation}
where $d\theta^2$ is the standard metric in $\mathbb S^{m-1}$, and
\begin{equation}\label{eq601}
\psi\in \mathcal A:=\Big\{f\in C^\infty((0,\infty))\cap
C^1([0,\infty)): f'(0)=1,\, f(0)=0,\, f>0\text{ in }
(0,\infty)\Big\}.
\end{equation}
In this case, we write $M\equiv M_\psi$; furthermore, we have
$\sqrt{A(r,\theta)}=\psi^{m-1}(r)$, so the boundary area of the
geodesic sphere $\partial S_R$ is computed by
\[S(R)=\omega_m\psi^{m-1}(R),\]
$\omega_m$ being the area of the unit sphere in $\mathbb R^m$.
Also, the volume of the ball $B_R(o)$ is given by
\[\mu(B_R(o))=\int_0^R S(\xi)d\xi\,. \]

Observe that for $\psi(r)=r$, $M=\mathbb R^m$, while for
$\psi(r)=\sinh r$, $M$ is the $m-$dimensional hyperbolic space
$\mathbb H^m$.

\begin{exe}\label{exe1}
Let $M$ be an $m-$dimensional model manifold with pole $o$ and
metric given by \eqref{eq600} with
\[\psi(r):=\begin{cases}
r  &  \,\,  \textrm{if} \  0\leq r<1  \, , \\
 [r^{\alpha-1}(\log r)^{\beta} ]^{\frac 1{m-1}} & \ \textrm{if} \  r > 2 \,; \\
\end{cases}
\]
where $\alpha>1$ and $\beta\in \left(0, \frac 1{q-1}\right]\,.$ We
consider problem \eqref{Eq} with $V\equiv 1$ and $p=2$.  Note that
for $R>0$ large enough
\[\mu(B_R)\simeq C R^{\alpha}(\log R)^\beta\leq C R^{\alpha+\sigma}\,,\]
for any $\sigma>0$, while
$$\lim_{R\to+\infty}\frac{\mu(B_R)}{R^\alpha}=+\infty.$$ Furthermore,
\[\frac{d}{dr}\pa{\log\sqrt{A(r)}}=\frac{d}{dr}\pa{\log\big([\psi(r)]^{m-1}\big)}\leq \frac C r\quad \textrm{for all}\,\, r>0\,.\]
Thus, for $\alpha>2$, from \cite[Theorem A]{Zhang} we can infer that
problem \eqref{Eq} does not admit nonnegative nontrivial solutions,
provided that $1<q\leq 1+ \frac{2}{\alpha+\sigma}$ for some
$\sigma>0$, that is provided that
\[1<q < 1 +\frac 2{\alpha}\,.\]

On the other hand, just assuming $\alpha>1$, we can apply Corollary
\ref{cor3} with $p=2$ (see also Remark \ref{rem2}), where
$f(t)\equiv 1$, $g(x)\equiv 1$, $\sigma_1=\alpha$, $\sigma_2=1$,
$\delta_1=\beta$, $\delta_2=0$, and thus we can deduce that problem
\eqref{Eq} does not admit nonnegative nontrivial solutions, provided
that
\[1<q \leq 1 +\frac 2{\alpha}\,.\] So, we can exclude existence of nontrivial solutions also in the particular
case when $q=1+\frac{2}{\alpha}$.
\end{exe}


\begin{thebibliography}{999}
\bibitem{BPT} C. Bandle, M. A. Pozio, A. Tesei, {\it  The Fujita exponent for the Cauchy problem in the hyperbolic space
}, J. Diff. Eq. {\bf 251 } (2011)\,, 2143--2163\,.

%\bibitem{Br} R. Brooks,  {\it A relation between growth and the spectrum of the Laplacian},   Math. Z. {\bf 178} (1981), 501--508\,.

\bibitem{DaMit} L. D'Ambrosio, V. Mitidieri, {\it A priori estimates, positivity results, and nonexistence theorems for quasilinear degenerate elliptic inequalities}\,,
Adv. Math. {\bf 224} (2010), 967--1020\,.

\bibitem{DaLu} L. D'Ambrosio, S. Lucente, {\it Nonlinear Liouville theorems for Grushin and Tricomi operators}, J. Diff. Eq. {\bf 193} (2003), 511--541\,.

\bibitem{Dav} \newblock E. B. Davies, "Heat Kernel and Spectral Theory", \newblock Cambridge University Press (1989).


\bibitem{Fujita} H.
Fujita, {\it On the blowing up of solutions of the Cauchy problem
for $u_t = \Delta u + u^{1+\alpha}$}, J. Fac. Sci. Univ. Tokyo
Sect. I {\bf 13} (1966), 109--124.

\bibitem{Fuj2} H. Fujita {\it On some nonexistence and nonuniqueness theorems for nonlinear parabolic equations} Proc. Symp. Pure Math., {\bf 18}, Amer. Math. Soc. (1968), 138--161\,.


\bibitem{Galak1} V. A. Galaktionov, {\it Conditions for the absence of global
solutions for a class of quasilinear parabolic equations}, Zh.
Vychisl. Mat. i Mat. Fiz. {\bf 22} (1982), 322--338.

\bibitem{Galak2} V. A. Galaktionov, {\it Blow-up for quasilinear heat equations with critical Fujitas
exponents}, Proc. Roy. Soc. Edinburgh Sect. A {\bf 124} (1994),
517--525.

\bibitem{GalakLev} V. A. Galaktionov, H. A. Levine, {\it A general approach to
critical Fujita exponents in nonlinear parabolic problems},
Nonlinear Anal. {\bf 34} (1998), 1005--1027.




%\bibitem{Gidas} B. Gidas, {\it Symmetry properties and isolated singularities of positive solutions of nonlinear elliptic equations}, Nonlinear partial differential equations in engineering and applied science (Proc. Conf., Univ. Rhode Island, Kingston, R.I., 1979), Dekker New York, Lecture Notes in Pure and Appl. Math. {\bf 54} (1980), 255--273\,.
%\bibitem{GidasSpruck} B. Gidas, J. Spruck, {\it Global and local behavior of positive solutions of nonlinear elliptic equations}, Comm. Pure Appl. Math., {\bf 34} (1981), 525--598\,.
%\bibitem{GilTru} D. Gilbarg, N. S. Trudinger, "Elliptic partial differential equations of second order" (2001), Springer-Verlag, Berlin\,.
\bibitem{Grig} A. Grigor'yan, {\it Analytic and geometric background of recurrence and nonexplosion of the Brownian motion on Riemannian manifolds}, Bull. Am. Math. Soc. {\bf 36} (1999),  135--249\,.
\bibitem{GrigKond} A. Grigor'yan, V. A. Kondratiev, {\it On the existence of positive solutions of semilinear elliptic inequalities on Riemannian manifolds}, In Around the research of Vladimir Maz'ya. II, volume 12 of Int. Math. Ser. (N. Y.), pages
203--218. Springer, New York, 2010.
\bibitem{GrigS} A. Grigor'yan, Y. Sun, {\it On non-negative solutions of the inequality $\Delta u + u^\sigma \leq 0$ on Riemannian manifolds}\,, Comm. Pure Appl. Math.
{\bf 67} (2014), 1336--1352\,.

\bibitem{Haya} K. Hayakawa, {\it On nonexistence of global solutions of some
semilinear parabolic differential equations}, Proc. Japan Acad.
{\bf 49} (1973), 503--505.

%\bibitem{Ichi1}  K. Ichihara, {\it Curvature, geodesics and the Brownian motion on a Riemannian manifold. I. Recurrence properties}\,, Nagoya Math. J. {\bf 87} (1982), 101--114\,.
%\bibitem{Ichi2} K. Ichihara, {\it Curvature, geodesics and the Brownian motion on a Riemannian manifold. II. Explosion properties},\, Nagoya Math. J. {\bf 87} (1982), 115--125\,

\bibitem{Kurta} V.V. Kurta, {\it On the absence of positive solutions to semilinear elliptic equations }\,, Tr. Mat. Inst. Steklova, {\bf 227} (1999), 162--169\,.

\bibitem{Lev} H. A. Levine, {\it The role of critical exponents in blowup
theorems}, SIAM Rev. {\bf 32} (1990), 262--288.

\bibitem{MMP1} P. Mastrolia, D. D. Monticelli, F. Punzo, {\it Nonexistence results for elliptic differential inequalities with a potential on Riemannian
manifolds}, Cal. Var. PDE, DOI:10.1007/s00526-015-0827-0 (to
appear)\,.

\bibitem{MasRigSet}
 P. Mastrolia, M. Rigoli, A.G. Setti, {\it Yamabe-type equations on complete, noncompact
 manifolds}, Progress in Mathematics 302, Birkh\"auser Verlag, Basel, 2012.

\bibitem{MitPohoz359} V. Mitidieri, S. I. Pohozev, {\it Absence of global positive solutions of quasilinear elliptic inequalities}, Dokl. Akad. Nauk, {\bf 359} (1998), 456--460\,.
\bibitem{MitPohoz227} V. Mitidieri, S.I. Pohozaev, {\it Nonexistence of positive solutions for quasilinear elliptic problems in $\mathbb R^N$},  Tr. Mat. Inst. Steklova, {\bf 227} (1999), 192--222\,.
\bibitem{MitPohozAbsence} V. Mitidieri, S.I. Pohozaev, {\it A priori estimates and the absence of solutions of nonlinear partial differential equations and inequalities}, Tr. Mat. Inst. Steklova, {\bf 234} (2001)\,,1--384\,.
\bibitem{MitPohozMilan} V. Mitidieri, S.I. Pohozaev, {\it Towards a unified approach to nonexistence of solutions for a class of differential inequalities}, Milan J. Math., {\bf 72} (2004),  129--162\,.
\bibitem{Mont} D.D. Monticelli, {\it Maximum principles and the method of moving planes for a class of degenerate elliptic linear operators }, J. Eur. Math. Soc., {\bf 12} (2010),
611--654\,.
\bibitem{PoTe} S.I. Pohozaev, A. Tesei, {\it Nonexistence of local solutions to semilinear partial differential inequalities}. Ann. Inst.
H. Poinc. Anal. Non Lin., {\bf 21} (2004), 487--502\,.
\bibitem{P1} F. Punzo, {\it Blow-up of solutions to semilinear parabolic equations on Riemannian manifolds with negative sectional curvature }, J. Math. Anal. Appl., {\bf 387} (2012),
 815--827\,.
\bibitem{P2} F. Punzo, {\it Global existence of solutions to the semilinear heat equation on Riemannian manifolds with negative sectional
curvature}, Riv. Mat. Univ. Parma {\bf 5} (2014), 113--138\,.
 \bibitem{PuTe} F. Punzo, A. Tesei, {\it On a semilinear parabolic equation with inverse-square potential}, Atti Accad. Naz. Lincei Cl. Sci. Fis. Mat. Natur. Rend. Lincei (9) Mat. Appl.
{\bf 21} (2010) 359--396\,.
\bibitem{Sun1} Y. Sun, {\it Uniqueness results for nonnegative solutions of semilinear inequalities on Riemannian manifolds}, J. Math. Anal. Appl. {\bf 419} (2014), 646--661\,.
\bibitem{Sun2} Y. Sun, {\it On nonexistence of positive solutions of quasilinear inequality on Riemannian manifolds}, preprint (2013), https://www.math.uni-bielefeld.de/sfb701/files/preprints/sfb13068.pdf\,.

\bibitem{Weiss} F. B. Weissler, {\it Local Existence and Nonexistence for Semilinear Parabolic Equations in $L^{p}$}, Indiana Univ. Math. J. {\bf 29} (1980), 79--102\,.

\bibitem{Weiss2} F. B. Weissler, {\it Existence and nonexistence of global solutions for a semilinear heat equation}, Israel J. Math.  {\bf 38} (1981), 29--40\,.



\bibitem{Zhang} Q. S. Zhang, {\it Blow-up results for nonlinear parabolic equations on
manifolds}, Duke Math. J. {\bf 97} (1999), 515--539\,.
\end{thebibliography}
\end{document}